\documentclass[12pt, reqno]{amsart}
\usepackage{amssymb,amsthm,amsfonts,amsmath}

\usepackage{hyperref}
\usepackage{mathrsfs}

\usepackage[dvips]{graphicx}

\newcommand\A{\mathbb A}
\newcommand\Y{\mathbb Y}
\newcommand\Z{\mathbb Z}
\newcommand\C{\mathbb C}
\newcommand\R{\mathbb R}

\newcommand\h{\bold h}
\newcommand\e{\bold e}
\newcommand\he{\bold{\hat e}}
\newcommand\p{\bold p}
\newcommand\HH{\bold H}

\newcommand\HE{\bold{\widehat E}}

\newcommand\al{\alpha}
\newcommand\be{\beta}
\newcommand\ga{\gamma}
\newcommand\Ga{\Gamma}
\newcommand\de{\delta}

\newcommand\ka{\kappa}
\newcommand\La{\Lambda}
\newcommand\la{\lambda}
\newcommand\si{\sigma}
\newcommand\tth{\theta}
\newcommand\epsi{\varepsilon}
\newcommand\om{\omega}
\newcommand\Om{\Omega}

\newcommand\wt{\widetilde}

\newcommand\Prob{\operatorname{Prob}}

\newcommand\const{\operatorname{const}}

\newcommand\Fun{{\operatorname{Fun}}}
\newcommand\Area{\operatorname{Area}}
\newcommand\kernel{\mathcal K}
\newcommand\shape{\operatorname{shape}}
\newcommand\gr{\operatorname{gr}}

\newcommand\up{\uparrow}
\newcommand\down{\downarrow}

\newcommand\ome{\boldsymbol\omega}
\newcommand\pd{\partial}
\newcommand\F{\mathcal{F}}

\newtheorem{theorem}{Theorem}[section]
\newtheorem{proposition}[theorem] {Proposition}
\newtheorem{corollary}[theorem]{Corollary}
\newtheorem{lemma}[theorem]{Lemma}

\theoremstyle{definition}
\newtheorem{definition}[theorem]{Definition}
\newtheorem{remark}[theorem]{Remark}

\numberwithin{equation}{section}

\begin{document}

\title[Infinite-dimensional diffusion processes]
{Anisotropic Young diagrams and \\infinite-dimensional diffusion processes\\
with the Jack parameter}

\author{Grigori Olshanski}
\address{Institute for Information Transmission
Problems, Bolshoy Karetny 19, 127994 Moscow GSP-4, Russia; \newline \indent
Independent University of Moscow}

\email{olsh2007@gmail.com}

\begin{abstract}
We construct a family of Markov processes with continuous sample
trajectories on an infinite-dimensional space, the Thoma simplex. The family
depends on three continuous parameters, one of which, the Jack parameter, is
similar to the beta parameter in random matrix theory. The processes arise in a
scaling limit transition from certain finite Markov chains, the so called
up-down chains on the Young graph with the Jack edge multiplicities. Each of
the limit Markov processes is ergodic and its stationary distribution is a
symmetrizing measure. The infinitesimal generators of the processes are
explicitly computed; viewed as selfadjoint operators in the $L^2$ spaces over
the symmetrizing measures, the generators have purely discrete spectrum which
is explicitly described.

For the special value $1$ of the Jack parameter, the limit Markov processes
coincide with those of the recent work by Borodin and the author (Prob. Theory
Rel. Fields 144 (2009), 281--318). In the limit as the Jack parameter goes to
$0$, our family of processes degenerates to the one-parameter family of
diffusions on the Kingman simplex studied  long ago by Ethier and Kurtz in
connection with some models of population genetics.

The techniques of the paper are essentially algebraic. The main computations
are performed in the algebra of shifted symmetric functions with the Jack
parameter and rely on the concept of anisotropic Young diagrams due to Kerov.

\medskip
\noindent {\it Keywords\/}: Diffusion processes; up-down Markov chains; Thoma's
simplex; Jack symmetric functions; Young diagrams; z-measures; Kerov
interlacing coordinates; shifted symmetric functions; Poisson--Dirichlet
distribution; Selberg integral

\medskip
\noindent {\it Mathematics Subject Classification\/} (2000): 60J60; 60C05;
60J10; 05E05
\end{abstract}

\thanks{At various stages of work, the present research was supported by the RFBR
grants 07-01-91209 and 08-01-00110, by the project SFB 701 (Bielefeld
University), and by the grant ``Combinatorial Stochastic Processes" (Utrecht
University).}

\maketitle

\newpage

\tableofcontents

\newpage

\section{Introduction}\label{1}

\subsection{Motivation and general description of the work}\label{1-1}

This work can be viewed as a continuation of the project started by Borodin and
myself, see \cite{BO6}, \cite{BO7}, \cite{BO8}, but it can be also read
independently. The goal of the project is to study Markovian stochastic
dynamics in certain infinite-dimensional models that originate from random
partitions. The main difference of this paper from the previous ones is in
introducing a new parameter $\tth$, the so-called Jack parameter, which is
analogous to the $\be$-parameter (inverse temperature) in the log-gas systems
($\tth=\be/2$). Exact statements of our results can be found in the next
subsection. Meanwhile, we would like to explain some ideas behind this work and
to describe our motivation.

Several classes of infinite-dimensional Markov processes are known. A large
part of the literature on those deals with interacting particle systems on a
lattice, and also with similar systems in $\R^d$ that are related to Gibbs
measures. Statistical mechanics serves as the main motivational source for such
models. Another source of infinite-dimensional Markov processes is population
genetics. A very interesting but not very well studied problem is construction
of dynamics for particle systems with nonlocal interaction of log-gas type;
such systems naturally arise as large $N$ limits of $N$-particle random matrix
type ensembles.

The model that we study in this paper is of a different origin --- it came up
in the asymptotic representation theory of symmetric groups. Nevertheless, it
turns out to be somewhat similar to log-gas systems on one hand, and on the
other hand it is closely connected to one of the well-known models from
population genetics \cite{EK1}.

Constructing Markov dynamics on an infinite-dimensional state space often
constitutes a nontrivial problem. For example, it may be difficult to assign
rigorous meaning to an intuitive definition of the infinitesimal generator of
the Markov process. See, e.~g., the paper by Spohn \cite{Sp}, where the problem
of the justification of the large $N$ limit transition for Dyson's log-gas
systems is discussed. In order to construct a Markov generator, one often uses
Dirichlet forms. This is a very effective yet technically demanding analytic
method. In this paper we follow a more direct approach in which analysis is
largely replaced by algebra and combinatorics. I hope that some ideas below may
be useful for studying the dynamics in log-gas systems as well.

Let us now describe (in very general terms) the model considered below. Denote
by $\Y_n$ the set of partitions of the natural number $n$. We identify
partitions $\la\in\Y_n$ and Young diagrams with $n$ boxes. For any
$n=1,2,\dots$ we introduce a probability distribution $M^{(n)}_{\tth,z,z'}$ on
$\Y_n$ that depends on three continuous parameters $\tth,z,z'$. The resulting
ensemble of random partitions can be compared to $N$-particle random matrix
ensembles; the role of the parameter $N$ is played by $n$. Let me not give an
exact expression for the weights $M^{(n)}_{\tth,z,z'}(\la)$ here, as it
requires a fairly long discussion (cf. \cite{BO5}). Instead, let me note that
this expression can be represented in the form very much reminiscent of the
joint probability density for random matrix $\be$-ensembles:
$$
M^{(n)}_{\tth,z,z'}(\la)=\const \exp(-2\tth\, W(\la_1,\dots,\la_\ell)),
$$
where $\ell$ is the length of the partition $\la$, and $W$ is a function on
partitions that can be split into the sum of one-particle and two-particle
``interaction potentials'',
$$
W(\la_1,\dots,\la_\ell)=\sum_{i=1}^\ell W_1(\la_i)+\sum_{1\le
i<j\le\ell}W_2(\la_i,\la_j).
$$
The two-particle potential on large distances is asymptotically equivalent to
the logarithmic one:
$$
W_2(\la_i,\la_j)\sim \log\frac1{\la_i-\la_j}\,, \quad \la_i-\la_j\gg0.
$$
The analogy between random partitions and random matrices that looked startling
10 years ago (see, e.g., \cite{BO1}, \cite{BO4}), is nowadays viewed as
commonplace \cite{Ok}.

The key property of the distributions $M^{(n)}_{\tth,z,z'}$ is the fact that
they are related to each other via a certain canonical chain of Markovian
transition functions $\Y_n\to\Y_{n-1}$ ($n=1,2,\dots$), that depend only on
$\tth$, and that are defined in terms of the Jack symmetric functions
corresponding to the parameter $\tth$. According to a general theorem proved in
\cite{KOO}, this coherency property implies the existence of the limiting
probability distribution
$$
M_{\tth,z,z'}=\lim_{n\to\infty} M^{(n)}_{\tth,z,z'}\,,
$$
which lives on the infinite-dimensional compact space \footnote{The appearance
of the {\it double\/} set of coordinates in $\Om$ is related to the fact that
Young diagrams are {\it two-dimensional\/} objects --- they have rows and
columns that play equal roles in our model.}
\begin{multline*}
 \Om=\{(\al;\be): \al=(\al_1,\al_2,\dots), \quad \be=(\be_1,\be_2,\dots),\\
\al_1\ge\al_2\ge\dots\ge0,\quad
\be_1\ge\be_2\ge\dots\ge0,\quad\sum_i\al_i+\sum_j\be_j\le1\}.
\end{multline*}
The space $\Om$ is called the {\it Thoma simplex\/}, and the limit measures
$M_{\tth,z,z'}$ are called the (boundary) {\it z-measures\/}.

The above-mentioned theorem from \cite{KOO} claims that there exists a
one-to-one correspondence $M\leftrightarrow\{M^{(n)}\}$ between the probability
distributions $M$ on $\Om$ and the {\it coherent families\/} $\{M^{(n)}\}$,
that is, sequences of probability distributions related by the Markovian
($\tth$-dependent) transition functions $\Y_n\to\Y_{n-1}$ mentioned above. In
fact, this theorem contains more: it provides a possibility of constructing a
{\it canonical\/} Markov dynamics on $\Om$ that preserves the given
distribution $M$.

The idea is the following. There exists a simple and natural way of
constructing, for any coherent family $\{M^{(n)}\}$, a sequence of reversible
Markov chains $\Y_n\to\Y_n$ with stationary distributions $M^{(n)}$. We call
those the {\it up-down chains\/}. Now it is natural to raise a question whether
the up-down chains converge, as $n\to\infty$, to a Markov process on $\Om$ that
has $M$ as its stationary distribution. \footnote{I got this idea from a
conversation with my friend Sergei Kerov that took place in the nineties. Kerov
had never written about up-down chains, and for the first time they seem to
have appeared in the literature in the paper by Fulman \cite{Fu1}, who used
them for different purposes (cf. also his subsequent publications \cite{Fu2},
\cite{Fu3}). Borodin and I implicitly applied the up-down chains to
constructing Markov dynamics in \cite{BO6} and \cite{BO7}, and then explicitly
in \cite{BO8}.} In the concrete case of z-measures $M$ we can answer this
question in the affirmative. \footnote{I do not know what can be done for
arbitrary measures on $\Om$.}

Although the results of this paper are stated in probabilistic terms, the main
content of the paper is algebraic, and it can be phrased as follows. Consider
the algebra $\Lambda$ of symmetric functions, and identify it with the algebra
of polynomials in countably many generators $p_1,p_2, \dots$ (the power sums).
Further, denote by $T_n$ the Markov operator for the $n$th up-down chain.
Originally $T_n$ is defined as an operator in the space of functions on the
finite set $\Y_n$. However, we show that there exists a uniformly defined (for
all $n=1,2,\dots$) representation of the operators $T_n$ in terms of certain
operators in the algebra $\Lambda$ (essentially we carry the operators $T_n$
over to a common space). The most difficult part of the work is the computation
of this representation in the form of a differential operator with respect to
formal variables $p_1,p_2, \dots$. The techniques that we apply here are
discussed in Subsection \ref{1-5} below.

After bringing $T_n$'s to a suitable form, we find the pre-generator of the
Markov process on $\Om$ as the limit (in a certain rigorous sense)
$$
A=\lim_{n\to\infty}n^2(T_n-1).
$$
The factor $n^2$ corresponds to scaling time --- one step of the Markov chain
with large number $n$ is equated to a small time interval of size $\Delta
t=n^{-2}$. The justification of the limit transition is performed using
standard techniques (Trotter-type theorems, see the book \cite{EK2}), as well
as some ideas from the paper \cite{EK1}. This paper and its relation to the
present work is further discussed in Subsection \ref{1-2} below.

\subsection{Ethier--Kurtz's diffusions and statement of the main
results}\label{1-2} In the remarkable paper \cite{EK1} published in 1981,
Ethier and Kurtz studied a one-parameter family of diffusions \footnote{By a
diffusion we mean a strong Markov process with continuous sample trajectories.}
on the space
\begin{equation}\label{tag1.1}
\overline\nabla_\infty=\{\al=(\al_1,\al_2,\dots): \al_1\ge \al_2\ge\dots\ge0,
\quad \sum_{i=1}^\infty \al_i\le1\}.
\end{equation}
The space $\overline\nabla_\infty$ is compact in the topology of coordinatewise
convergence and can be regarded as an infinite-dimensional simplex. We call it
the {\it Kingman simplex\/}. The diffusions are determined by infinitesimal
generators acting on an appropriate space of functions on
$\overline\nabla_\infty$ and can be written as second order differential
operators
\begin{equation}\label{tag1.2}
\sum_{i,j=1}^\infty \al_i(\de_{ij}-\al_j)\frac{\pd^2}{\pd \al_i\pd
\al_j}-\tau\sum_{i=1}^\infty \al_i\frac{\pd }{\pd \al_i}\,,
\end{equation}
where $\tau>0$ is the parameter. \footnote{In \cite{EK1}, the parameter is
denoted as $\tth$. We use another symbol because of a conflict of notations. We
have also omitted the factor of $\frac12$ in the formula of \cite{EK1}.} Each
of the diffusions has a unique stationary distribution and is reversible and
ergodic. This is a nice example of infinite-dimensional Markov processes,
especially interesting because the stationary distributions are the famous
Poisson-Dirichlet distributions (about them, see, e.g. \cite{Ki2}).

The 3-parameter family of Markov processes on the Thoma simplex $\Om$,
constructed in the present paper, is a wider model of infinite-dimensional
diffusions, containing Ethier--Kurtz's diffusions as a limit case. In our
model, in contrast to that of Ethier--Kurtz, the infinitesimal generator cannot
be written as a differential operator in natural coordinates $(\al;\be)$ on the
state space $\Om$. Nevertheless, the restriction of the generator on an
appropriate invariant core $\F$ admits an explicit expression.

Specifically, the core $\F$ is the algebra of polynomials $\R[q_1,q_2,\dots]$,
where $q_1,q_2,\dots$ are the following functions on $\Om$
\begin{equation}\label{tag1.3}
q_k(\al;\be)=\sum_{i=1}^\infty \al_i^{k+1}+(-\tth)^k \sum_{i=1}^\infty
\be_i^{k+1}, \qquad k=1,2,\dots,
\end{equation}
and $\tth>0$ is the Jack parameter mentioned above. \footnote{The functions
\eqref{tag1.3} are the Jack deformation of the {\it supersymmetric\/} power
sums in coordinates $\al_i$ and $-\be_j$, cf. \cite{Ma}. About the link between
supersymmetry and Young diagrams, see \cite{VK}, \cite{KeO}, \cite{KOO}.} These
functions are continuous and algebraically independent. We call them the {\it
moment coordinates\/} for the following reason: Let us embed $\Om$ into the
space of probability measures on the closed interval $[-\tth,1]\subset\R$ by
assigning to an arbitrary point $(\al;\be)\in\Om$ the atomic measure
\begin{equation}\label{tag1.4}
\nu_{\al;\be}=\sum_{i=1}^\infty\al_i\de_{\al_i}+
\sum_{i=1}^\infty\be_i\de_{-\tth\be_i} + \ga\de_0, \qquad
\ga:=1-\sum\al_i-\sum\be_i\,,
\end{equation}
where $\de_x$ stands for the Dirac measure at $x\in\R$. Then $q_k(\al;\be)$ is
equal to the $k$th moment of the measure $\nu_{\al;\be}$.

Even though the moment coordinates are not true coordinates on $\Om$ in the
conventional differential-geometric sense, they allow us to define the Markov
pre-generator as a second order differential operator acting in the polynomial
algebra $\F$ and depending on the Jack parameter $\tth$ and two additional
continuous parameters $z$ and $z'$:
\begin{multline}\label{tag1.5}
\sum_{i,j\ge1}(i+1)(j+1)
(q_{i+j}-q_i q_j)\frac{\pd^2}{\pd q_i \pd q_j}\\
+\sum_{i\ge1}(i+1)\big[((1-\tth)i+(z+z')) q_{i-1}
-(i+\tth^{-1}zz')q_i\big]\frac{\pd}{\pd q_i}\\
+\tth\sum_{i,j\ge0}(i+j+3)q_i q_j\frac{\pd}{\pd q_{i+j+2}}\,,
\end{multline}
where $q_0\equiv1$. Now are in a position to state the main results of the
paper.

\begin{theorem}\label{1.1}
Let $C(\Om)$ be the Banach space of continuous real-valued
functions on the Thoma simplex\/ $\Om$. Regard the differential operator
\eqref{tag1.5} as an operator in $C(\Om)$ with dense invariant domain
$\F\subset C(\Om)$.

{\rm(i)} The operator \eqref{tag1.5} is closable in  $C(\Om)$ and its closure
serves as the infinitesimal generator of a diffusion process on $\Om$.

{\rm(ii)} The process has a unique stationary distribution and is reversible
and ergodic.

{\rm(iii)} The closure of the pre-generator \eqref{tag1.5} in the $L^2$ space
with respect to the stationary distribution is a self-adjoint operator with
purely discrete spectrum
$$
\{0\}\cup\{-m(m-1+\tth^{-1}zz'): m=2,3,\dots\},
$$
where the multiplicity of the eigenvalue $-m(m-1+\tth^{-1}zz')$ equals the
number of partitions of $m$ without parts equal to $1$.

{\rm(iv)} The stationary distribution is the z-measure $M_{\tth,z,z'}$.
\end{theorem}

About the z-measures see the next subsection. The restrictions on parameters
$(z,z')$ are indicated below in Proposition \ref{5.3}.

In the limit regime as
\begin{equation}\label{tag1.6}
z\to0, \quad z'\to0, \quad \tth\to0,\quad \tth^{-1}zz'\to\tau>0,
\end{equation}
our model degenerates to the Ethier--Kurtz model with parameter $\tau$. Let me
explain this informally:

Observe that as $\tth\to0$, the interval $[-\tth,1]$ shrinks to $[0,1]$ so that
the $\be$-coordinates disappear from \eqref{tag1.4}. This explains why the
Thoma simplex $\Om$ degenerates to the Kingman simplex \eqref{tag1.1}. Next, in
the regime \eqref{tag1.6}, the expression \eqref{tag1.5} degenerates to
\begin{equation}\label{tag1.7}
\sum_{i,j\ge1}(i+1)(j+1) (q_{i+j}-q_i q_j)\frac{\pd^2}{\pd q_i \pd q_j}\;+\;
\sum_{i\ge1}(i+1)\big[i q_{i-1}-(i+\tau) q_i\big]\frac{\pd}{\pd q_i}\,.
\end{equation}
Finally, it is not difficult to show that \eqref{tag1.7} is merely another form
of \eqref{tag1.2}, provided that the moment coordinates in \eqref{tag1.7} are
viewed as functions on the subspace
$\{(\al,0)\subset\Om\}=\overline\nabla_\infty$. \footnote{The fact was
established in \cite{Sch}, see also \cite{Pe1}.} Note also that in the regime
\eqref{tag1.6}, the z-measures weakly converge to the Poisson-Dirichlet
distribution with parameter $\tau$.

\subsection{Z-measures and Selberg integrals}\label{1-3}
The z-measures form a distinguished family of probability measures on the Thoma
simplex $\Om$. Theses measures first emerged in the note \cite{KOV1} in
connection with the problem of harmonic analysis on the infinite symmetric
group, see also \cite{KOV2} for a detailed exposition and \cite{B2},
\cite{BO1}, \cite{BO2}. All these papers concerned the special case $\tth=1$.
Then the z-measures are the spectral measures governing the decomposition of
some analogs of the regular representation; here it is worth noting that $\Om$
is a kind of dual space to the infinite symmetric group. The case
$\tth=\frac12$ is also related to a problem of harmonic analysis (see
\cite{Str}), while (as was already mentioned above) the limit case $\tth=0$
corresponds to the Poisson-Dirichlet distributions. For general $\tth>0$, the
z-measures were defined in \cite{Ke4} (see also \cite{BO3} for a different
approach). The idea of building a theory valid for all $\tth>0$ is similar to
Dyson's idea of introducing the beta parameter into random matrix theory (see
\cite{Dy}) \footnote{In our notation, Dyson's $\be$ corresponds to $2\tth$. }
or to Heckman--Opdam's idea of generalizing harmonic analysis on symmetric
spaces to root systems with formal root multiplicities (see Heckman's lectures
in \cite{HS}).

A formal definition of the z-measures can be given in the following way.
Consider the algebra $\La$ of symmetric functions and its basis $\{\mathcal
P_{\la;\tth}\}$ of Jack symmetric functions with parameter $\tth$; here the
index $\la$ ranges over the set $\Y$ of Young diagrams. \footnote{See, e.g.,
\cite{Ma}. Our parameter $\tth$ is inverse to Macdonald's parameter $\al$ and
coincides with Kadell's \cite{Ka} parameter $k$.} Identify $\La$ with
$\R[p_1,p_2,\dots]$ where $p_k$'s are the Newton power sums, and consider the
algebra morphism $\La\to C(\Om)$ defined by
$$
p_1\to1, \quad p_2\to q_1\,, \quad p_3\to q_2\,, \dots,
$$
where $q_1=q_1(\al;\be)$, $q_2=q_2(\al;\be)$,\dots are the moment coordinates
\eqref{tag1.3}. Then each Jack function $\mathcal P_{\la;\tth}$ turns into a
continuous function $\mathcal P^\circ_{\la;\tth}(\al;\be)$ on $\Om$, which may
be viewed as a version of {\it supersymmetric\/} Jack function in $\al$ and
$\be$. \footnote{About applications of such functions to integrable systems,
see \cite{SV}.} The z-measure with parameters $(\tth,z,z')$, denoted as
$M_{\tth,z,z'}$, can be characterized by the integrals
\begin{equation}\label{tag1.8}
M^{(n)}_{\tth,z,z'}(\la):=[p_1^n:\mathcal P_{\la;\tth}]\,\int_\Om \mathcal
P^\circ_{\la;\tth}(\al;\be)M_{\tth,z,z'}(d\al\,d\be)
\end{equation}
for which there is a nice multiplicative formula (here $n$ equals $|\la|$, the
number of boxes in the diagram $\la$, and $[p_1^n:\mathcal P_{\la;\tth}]$
stands for the coefficient of $\mathcal P_{\la;\tth}$ in the expansion of
$p_1^n$ in the basis of Jack functions). For any fixed $\la$,
$M^{(n)}_{\tth,z,z'}(\la)$ is a rational function in $z$, $z'$, and $\tth$,
whose explicit expression can be found in \cite{BO5}.

For any fixed $\tth>0$, the set of admissible values of $(z,z')$ (for which the
z-measure is well defined as a probability measure on $\Om$) is precisely the
set of those $(z,z')$ for which the quantities $M^{(n)}_{\tth,z,z'}(\la)$ are
nonnegative for all diagrams $\la$ (about the meaning of this condition, see
the next subsection). Since $M^{(n)}_{\tth,z,z'}(\la)$ does not change under
transposition $z\leftrightarrow z'$, one can equally well take as parameters
$z+z'$ and $zz'$; in these new coordinates, the set of admissible values
becomes a closed subset of $\R^2$ with a nonempty interior. This set can be
divided into two parts depending on whether the integrals \eqref{tag1.8} are
strictly positive for all $\la$ (the nondegenerate series) or vanish for some
$\la$ (the degenerate series). For more detail, see \cite{BO5}.

The present paper focuses on the nondegenerate series but I would like to give
some comments on the degenerate series, because the degenerate z-measures
demonstrate in miniature some features of the general z-measures.

The characteristic property of the degenerate series is that the corresponding
z-measure $M_{\tth,z,z'}$ is supported by a finite-dimensional subset in $\Om$.
The simplest example is
$$
z=N\tth, \quad z'=(N-1)\tth+\si, \qquad N=1,2,\dots, \quad \si>0.
$$
Then the support of the z-measure is the subset
\begin{multline*}
\{(\al;\be): \al_{N+1}=\al_{N+2}=\dots=0, \quad\be_1=\be_2=\dots=0, \\
\al_1\ge\dots\ge\al_N\ge0, \quad \al_1+\dots+\al_N=1\}\subset\Om,
\end{multline*}
which is a simplex of dimension $N-1$; the measure itself has the form
\begin{equation}\label{tag1.9}
\const\cdot \prod_{i=1}^N \al_i^{\si-1}\cdot\prod_{1\le i<j\le
N}(\al_i-\al_j)^{2\tth}\cdot d^\circ\al,
\end{equation}
where $d^\circ\al$ stands for the Lebesgue measure on the simplex.

It is worth noting that in this special case, the integrals \eqref{tag1.8} turn
into Selberg-type integrals, see \cite{Ke3}, \cite{Ke6}. More refined examples
of degenerate z-measures involve both the $\al$ and $\be$ coordinates and
provide super-analogs of the Selberg integral (see \cite{BO3}). In a certain
sense, the general z-measures on $\Om$ can be viewed as an infinite-dimensional
(super) generalization of the Selberg measures \eqref{tag1.9}. \footnote{About
the history and various aspects and versions of the Selberg integral, see the
recent survey \cite{FW}.}

Closely related to \eqref{tag1.9} is the following probability measure on the
cone $\al_1\ge\dots\ge\al_N\ge0$ in $\R^N$
\begin{equation}\label{tag1.10}
\const\cdot \prod_{i=1}^N \al_i^{\si-1}e^{-\al_i}\cdot\prod_{1\le i<j\le
N}(\al_i-\al_j)^{2\tth}\cdot d\al,
\end{equation}
where $d\al$ is the Lebesgue measure on the cone. \footnote{The passage from
\eqref{tag1.9} to \eqref{tag1.10} is similar to that from Euler's Beta integral
to Euler's Gamma integral.} Note that \eqref{tag1.10} determines the
$N$-particle {\it Laguerre ensemble\/} with the beta parameter $2\tth$.

This example builds a bridge between the z-measures and random matrix type
ensembles and  illustrates the thesis that the Jack parameter $\tth$ plays the
role of the beta parameter of random matrix theory.

\subsection{Discrete approximation. Up-down Markov chains}\label{1-4}
As was already mentioned above, the diffusion processes of Theorem \ref{1.1}
arise as limits of some finite Markov chains. Here is an outline of the
construction.

Recall that by $\Y_n$ we denote the set of Young diagrams with $n$ boxes,
$n=0,1,2,\dots$\,. Let us return to integrals \eqref{tag1.8} and write them in
abstract form
\begin{equation}\label{tag1.11}
M^{(n)}(\la):=[p_1^n:\mathcal P_{\la;\tth}]\,\int_\Om \mathcal
P^\circ_{\la;\tth}(\al;\be)M(d\al\,d\be), \quad \la\in\Y_n\,,
\end{equation}
where $M$ is an arbitrary probability measure on $\Om$. One can prove that the
functions $\mathcal P^\circ_{\la;\tth}(\al;\be)$ are nonnegative on $\Om$,
which implies that the numbers $M^{(n)}(\la)$ are nonnegative, too.
Furthermore, for any fixed $n$, one has
$$
\sum_{\la\in\Y_n}M^{(n)}(\la)=1,
$$
so that $M^{(n)}(\,\cdot\,)$ is a probability measure on the finite set $\Y_n$.
Next, one can prove that there exist embeddings $\iota_{\tth,n}:\Y_n\to\Om$
such that the push-forwards $\iota_{\tth,n}(M^{(n)})$ converge to $M$ in the
weak topology. It is worth noting that the embeddings do not depend on the
initial measure $M$.

Thus, for each fixed $\tth>0$, there exists an approximation of the compact
space $\Om$ by finite sets $\Y_n$ providing an approximation of any probability
measure $M$ on $\Om$ by some canonical sequence $\{M^{(n)}\}$  of probability
measures on the sets $\Y_n$. \footnote{In the special case $\tth=1$ this fact
is a refinement of Thoma's theorem \cite{T}, essentially due to Vershik and
Kerov \cite{VK}. For general $\tth>0$ this is the result of \cite{KOO}. A
similar fact holds in the limit case $\tth=0$, with $\Om$ replaced by
$\overline\nabla_\infty$: this is Kingman's theorem, see \cite{Ki1},
\cite{Ke6}.}

The sequences $\{M^{(n)}\}$ coming from probability measures on $\Om$ can be
characterized by a system of relations,
\begin{equation}\label{tag1.12}
M^{(n-1)}(\mu)=\sum_{\la\in\Y_n}M^{(n)}(\la)p^\down_\tth(\la,\mu),
\end{equation}
where $n$ and $\mu\in\Y_{n-1}$ are arbitrary and
\begin{equation}\label{tag1.13}
p^\down_\tth(\la,\mu):=\frac{[p_1^{n-1}:\mathcal P_{\mu;\tth}][p_1 \mathcal
P_{\mu;\tth}:\mathcal P_{\la;\tth}]}{[p_1^n:\mathcal P_{\la;\tth}]}.
\end{equation}

Thus, \eqref{tag1.11} establishes a bijective correspondence between
probability measures $M$ on $\Om $ and sequences $\{M^{(n)}\}$ of probability
measures on the sets $\Y_n$, satisfying the relations \eqref{tag1.12}. Such
sequences are called {\it coherent systems\/}. \footnote{We follow the
terminology adopted in \cite{Ol}. In the limit case $\tth=0$, coherent systems
were earlier introduced by Kingman under the name of {\it partition
structures\/}, see \cite{Ki1}.}

The quantities \eqref{tag1.13} are nonnegative and, for fixed $\la\in\Y_n$,
$$
\sum_{\mu\in\Y_{n-1}}p^\down_\tth(\la,\mu)=1,
$$
so that they determine transition kernels from $\Y_n$ to $\Y_{n-1}$, for each
$n$. We call them the {\it down transition probabilities\/} (these are the
Markovian transition functions mentioned above in Subsection \ref{1-1}). The
relations \eqref{tag1.12} mean that the down transition kernels transform
$M^{(n)}$ to $M^{(n-1)}$, for each $n$.

Now assume that $M$ is nondegenerate in the sense that for the corresponding
coherent system $\{M^{(n)}\}$, all quantities $M^{(n)}(\la)$ are strictly
positive. This condition is fulfilled, for instance, if the topological support
of $M$ is the whole space $\Om$. Then one can define some {\it up transition
probabilities\/} $p^\up_{\tth,M}(\la,\nu)$ which determine transition kernels
in the reverse direction, from $\Y_n$ to $\Y_{n+1}$, and transform $M^{(n)}$ to
$M^{(n+1)}$, for each $n$. Let us emphasize that the up transition
probabilities depend not only on $\tth$ (as the down probabilities) but also of
the coherent system, that is, of the initial measure $M$.

Taking the superposition of these two transition kernels we get, for each $n$,
a transition kernel from $\Y_n$ to itself,
\begin{equation}\label{tag1.14}
\Prob_n\{\la\to\ka\}=\sum_{\nu\in\Y_{n+1}}p^\up_{\tth,M}(\la,\nu)p^\down_\tth(\nu,\ka),
\quad \la,\ka\in\Y_n\,,
\end{equation}
It determines a reversible ergodic Markov chain on $\Y_n$ which has $M^{(n)}$
as the stationary distribution. We call this chain the $n$th {\it up-down
Markov chain\/} associated with $M$.

Thus, given $M$, we dispose not only of a canonical approximation $M^{(n)}\to
M$ but also of a natural reversible Markov chain preserving the $n$th measure,
for each $n$. This fact forms the basis of the work: the idea is to analyze the
asymptotics of the up-down chains associated with a nondegenerate z-measure, as
$n\to\infty$, and show that the chains have a scaling limit leading to a Markov
process on $\Om$.

This idea was first realized for the special case $\tth=1$ in the paper
\cite{BO8} by Borodin and myself. \footnote{Note that Ethier--Kurtz's
diffusions were obtained in two ways, both using an approximation procedure,
but the motivation of \cite{EK1} was quite different. Petrov \cite{Pe1} showed
that application of up-down chains makes it possible to re-derive
Ethier--Kurtz's diffusions and also get their analogs corresponding to Pitman's
two-parameter generalization of the Poisson-Dirichlet distributions. } However,
our computation of the limit pre-generator in the moment coordinates relied on
a combinatorial result of Lascoux and Thibon \cite{LT} for which no Jack analog
is available. In the present paper I apply another method; its basic ideas and
related concepts are described in the next subsection.

\subsection{Shifted symmetric functions, interlacing coordinates, and
anisotropic diagrams}\label{1-5} Let $\Fun(\Y_n)$ be the space of functions on
the finite set $\Y_n$ and $T_n:\Fun(\Y_n)\to\Fun(\Y_n)$ be the one-step
operator of the $n$th up-down Markov chain, induced by the transition kernel
\eqref{tag1.14}:
$$
(T_nF)(\la)=\sum_{\ka\in\Y_n}\Prob_n\{\la\to\ka\}F(\ka), \quad \la\in\Y_n\,.
$$
We show that
\begin{equation}\label{tag1.15}
\lim_{n\to\infty}n^2(T_n-1)=A,
\end{equation}
where $A$ is the differential operator \eqref{tag1.5}. Although the pre-limit
operators live in varying spaces, one can give a sense to the limit transition
by making use of the projections $C(\Om)\to\Fun(\Y_n)$, which are induced by
the embeddings $\iota_{\tth,n}:\Y_n\to\Om$ mentioned above. Here we employ a
well-known formalism, described in \cite{EK2}. Then, using a refined version of
Trotter's theorem (\cite[Theorem 7.5]{EK2}), we show that $A$ is closable and
generates a Markov semigroup in $C(\Om)$. The remaining claims of Theorem
\ref{1.1} are established in the same way as in \cite{BO8}.

The heart of the paper is the proof of \eqref{tag1.15}. To handle the
transition probabilities entering formula \eqref{tag1.14} we use an ingenious
trick invented by Kerov \cite{Ke4}: Kerov's idea was to consider {\it
anisotropic Young diagrams\/} made of rectangular boxes of size $\tth\times1$
and to parametrize such diagrams by pairs of interlacing sequences, which
encode the positions of the outer and inner corners (for more detail, see
Section \ref{4}). This trick allows one to completely avoid the hard machinery
related to Jack symmetric functions and reduce the proof of \eqref{tag1.15} to
a computation in the algebra of {\it $\tth$-regular functions\/} on the set
$\Y$ of Young diagrams.

This algebra, denoted as $\A_\tth$, consists of {\it $\tth$-shifted symmetric
functions\/} in coordinates $\la_1,\la_2,\dots$\,. According to the original
definition of the algebra $\A_\tth$ (see \cite{KOO}), it is generated by the
``$\tth$-shifted'' analogs of power sums
$$
p^*_m(\la)=\sum_{i=1}^\infty [(\la_i-\tth i)^m-(-\tth i)^m], \qquad
m=1,2,\dots, \quad \la\in\Y.
$$
On the other hand, an important fact is that $\A_\tth$ also admits a nice
description in terms of Kerov's interlacing coordinates.

Note also that the computation of the limit operator in \eqref{tag1.15}
substantially employs an asymptotic formula for $\tth$-regular functions,
established in \cite{KOO} (see Theorem \ref{9.5} below).

\subsection{A variation: shifted Young diagrams and Schur's
Q-func\-tions}\label{1-6} Schur's Q-functions span a proper subalgebra in the
algebra of symmetric functions. As well known, these functions play the same
role in the theory of projective characters of symmetric groups as the ordinary
Schur functions do for ordinary characters. An analog of z-measures related to
Schur's Q-functions was found in \cite{B1}, see also \cite{BO3}. Replacing the
ordinary Young diagrams by the so-called shifted Young diagrams (which
correspond to strict partitions), one can define again the up-down Markov
chains. Their scaling limits were studied by Petrov \cite{Pe2}. The results he
obtained are parallel to those of \cite{BO8}, but the computation leading to an
analog of formula \eqref{tag1.5} for the pre-generator is based on the method
of the present paper.

\subsection{Organization of the paper}\label{1-7}
In Section \ref{2} we discuss the general formalism of up-down Markov chains.
In Section \ref{3} we recall the definition of the Young graph with Jack edge
multiplicities \cite{KOO} and introduce the corresponding system of down
probabilities. In Section \ref{4} we explain what are Kerov's anisotropic Young
diagrams and their interlacing coordinates \cite{Ke4}. Using these concepts, we
give an alternative definition of the down probabilities, and then in Section
\ref{5} we describe the up probabilities associated to the z-measures. In
Section \ref{6} we present the necessary material about the algebra $\A_\tth$
of $\tth$-regular functions on $\Y$. Here we also establish a link between
$\A_\tth$ and the up and down transition functions. The long Section \ref{7}
contains the key computation. Its result, which is stated in the beginning of
the section (Theorem \ref{7.1}), describes the top degree terms of the down and
up operators in the algebra $\A_\tth$. Proceeding from this computation, we
find in Section \ref{8} the top degree term of the operator $T_n-1$ (Theorem
\ref{8.2}). Combining this with an asymptotic theorem from \cite{KOO} we
perform in Section \ref{9} the limit transition from the up-down Markov chains
to diffusion processes on $\Om$: the final results are Theorems \ref{9.6},
\ref{9.7}, \ref{9.9}, and \ref{9.10}.

\section{Markov growth of Young diagrams and associated up-down Markov
chains}\label{2}

Let $\Y$ denote the set of all Young diagrams, including the empty diagram
$\varnothing$, and let $\Y_n\subset\Y$ be the subset of diagrams with $n$
boxes, $n=0,1,2,\dots$. Thus, $\Y$ is the disjoint union of the finite sets
$\Y_0$, $\Y_1$, \dots . By $|\la|$ we denote the number of boxes in a diagram
$\la$. As in \cite{Ma}, we identify Young diagrams and the corresponding
partitions of natural numbers, so that $\Y_n$ is identified with the set of
partitions of $n$. Using this identification we write Young diagrams in the
partition notation: $\la=(\la_1,\la_2,\dots)$.

If $\la$ and $\mu$ are two Young diagrams then we write $\mu\nearrow\la$ or,
equivalently, $\la\searrow\mu$ if $\mu\subset\la$ and $|\la|=|\mu|+1$ (that is,
$\mu$ is obtained from $\la$ by removing a box).

The {\it Young graph\/} is the graph with the vertex set $\Y$ and the edges
formed by arbitrary couples of diagrams, $\mu$ and $\la$, such that
$\mu\nearrow\la$. This is a {\it graded graph\/}, in the sense that the vertex
set $\Y$ is partitioned into {\it levels\/} (the finite sets $\Y_n$) and only
vertices of adjacent levels can be joined by an edge.

By an {\it infinite standard Young tableau\/} we mean an infinite sequence of
Young diagrams, $\{\la(n)\}_{n=0,1,2,\dots}$, subject to the following
condition: for any $n$, one has $\la(n)\in\Y_n$ and $\la(n)\nearrow\la(n+1)$.
In other words, this is an infinite monotone path in the Young graph started at
$\varnothing\in\Y_0$. Let $\mathcal T$ denote the space of all infinite
standard Young tableaux; it is a closed subset in the infinite product space
$\prod\Y_n$ equipped with the product topology. Thus, $\mathcal T$ is a compact
topological space and we can define the sigma-algebra of Borel subsets in
$\mathcal T$.

Assume we are given a probability Borel measure $\mathcal M$ on the space
$\mathcal T$. Then $\mathcal M$ can be viewed as the law of a random sequence
$\{\la(n)\}$ of Young diagrams. Let us say that $\mathcal M$ is a {\it Markov
measure\/} if $\{\la(n)\}$ possesses the Markov property. That is, conditioned
on $\la(n)=\la$, the subsequences $\{\la(0),\dots,\la(n-1)\}$ and
$\{\la(n+1),\la(n+2),\dots\}$ are independent from each other.

\begin{definition}[\rm Up and down transition probabilities]\label{2.1}
With any
Markov measure $\mathcal M$ on $\mathcal T$ we associate the following objects:
the {\it one-dimensional distributions\/} $M^{(n)}$, the {\it up transition
probabilities\/} $p^\up(\la,\nu)$, and the {\it down transition
probabilities\/} $p^\down(\la,\mu)$.  Here $M^{(n)}$ is the probability measure
on $\Y_n$ defined by
$$
M^{(n)}(\la)=\Prob\{\la(n)=\la\}, \qquad \la\in\Y_n\,.
$$
Further, for $\la\in\Y_n$, $\nu\in\Y_{n+1}$, and $\mu\in\Y_{n-1}$, we define
$p^\up(\la,\nu)$ and $p^\down(\la,\mu)$ as the conditional probabilities
\begin{gather*}
p^\up(\la,\nu)=\Prob\{\la(n+1)=\nu\mid \la(n)=\la\}, \\
p^\down(\la,\mu)=\Prob\{\la(n-1)=\mu\mid \la(n)=\la\}.
\end{gather*}
We view these probabilities as certain quantities associated to the oriented
edges of the graph.
\end{definition}

More precisely, the above definition makes sense if $M^{(n)}(\la)>0$ for all
$n$ and all $\la\in\Y_n$. This assumption holds in the concrete situation
studied in the paper. Note, however, that even if $M^{(n)}(\la)$ vanishes for
some diagrams $\la$, one can still define the up and down transition
probabilities on an appropriate subgraph of $\Y$.

Obviously, for any fixed $\la$,
$$
\sum_{\nu:\,\nu\searrow\la}p^\up(\la,\nu)=1, \qquad
\sum_{\mu:\,\mu\nearrow\la}p^\down(\la,\mu)=1,
$$
and the measures $M^{(n)}$ are consistent with both the up and down transition
probabilities in the following sense:
\begin{gather}
M^{(n+1)}(\nu)=\sum_{\la:\,\la\nearrow\nu}M^{(n)}(\la)p^\up(\la,\nu),
\label{tag2.1}\\
M^{(n-1)}(\mu)=\sum_{\la:\,\la\searrow\mu}M^{(n)}(\la)p^\down(\la,\mu).\label{tag2.2}
\end{gather}

Remark that the up transition probabilities $p^\up(\la,\nu)$ determine the
initial Markov measure $\mathcal M$ uniquely. Indeed, this happens because
there exists an initial ``time moment'', $n=0$, and the state space for $n=0$
is a singleton. Once we know the up transition probabilities, we can
reconstruct from the recurrence \eqref{tag2.1} the one-dimensional marginals
$M^{(n)}$ and, more generally, all finite-dimensional distributions. The up
transition probabilities are well suited to represent $\{\la(n)\}$ as a model
of {\it random Markov growth\/} of Young diagrams, where at each consecutive
moment of time a single new box is appended.

The down transition probabilities $p^\down(\la,\mu)$ do not possess the above
property for the obvious reason that for reversed time $n$, which ranges from
$+\infty$ to $0$, there is no finite initial moment. In such a situation, for a
given system of transition probabilities, a host of Markov measures satisfying
the corresponding recurrence relations may exist. Specifically, the following
abstract theorem holds:

\begin{theorem}\label{2.2} Fix an arbitrary system $p^\down=\{p^\down(\la,\mu)\}$
of down transition probabilities on the edges of the Young graph. That is,
assign to all downward oriented edges $\la\searrow\mu$  nonnegative numbers
$p^\down(\la,\mu)$ in such a way that
$$
\sum_{\mu:\,\mu\nearrow\la}p^\down(\la,\mu)=1 \quad\text{\rm for any fixed
vertex $\la$.}
$$

Then there exists a topological space $\Om(p^\down)$ and a function
$\kernel(\la,\om)$ on $\Y\times\Om(p^\down)$, continuous with respect to $\om$,
taking values in $[0,1]$, and such that the relation
\begin{equation}\label{tag2.3}
M^{(n)}(\la)=\int_{\Om(p^\down)}\kernel(\la,\om)M(d\om), \qquad \la\in\Y_n\,,
\quad n=1,2,\dots,
\end{equation}
establishes a bijective correspondence $\{M^{(n)}\}\leftrightarrow M$ between
sequences of probability measures solving the recurrence relations
\eqref{tag2.2} and probability measures on the space $\Om(p^\down)$.
\end{theorem}

This theorem is no more than an adaptation of well-known results concerning
boundaries of Markov chains (or more  general Markov processes). For a proof of
the theorem, see \cite{KOO}. For the concrete systems $p^\down$ considered in
the present paper, the precise form of the space $\Om(p^\down)$ and the kernel
$\kernel(\la,\om)$ is indicated in Subsection \ref{9-4} below.

The space $\Om(p^\down)$ is called the (minimal) {\it entrance boundary\/} for
the couple $(\Y,p^\down)$, and for the measure $M$ we will use the term the
{\it boundary measure\/} of $\{M^{(n)}\}$. (Note that the marginals $M^{(n)}$
together with the down transition probabilities already suffice to reconstruct
the initial Markov measure $\mathcal M$.)

We will not use Theorem \ref{2.2} in our arguments but it is useful for better
understanding the constructions of the paper. Heuristically, the result of
Theorem \ref{2.2} can be explained as follows: If we be dealing with {\it
finite\/} Markov sequences $\{\la(0),\dots,\la(n)\}$, then we could reconstruct
the law of such a sequence from its down transition probabilities and the
distribution $M^{(n)}$ on the uppermost level $\Y_n$. For infinite sequences,
the boundary $\Om(p^\down)$ plays the role of the nonexisting uppermost level
$\Y_\infty$, and the boundary measure $M$ is a substitute of the nonexisting
distribution $M^{(\infty)}$. It is not surprising that $\Om(p^\down)$ is
obtained as a kind of limit of the sets $\Y_n$ as $n\to\infty$. As for $M$,
then at least for concrete down transition probabilities that are discussed
below, $M$ can also be obtained as a limit of the distributions $M^{(n)}$.

We will regard down transition probabilities as a tool for specifying a {\it
class\/} of Markov measures on $\mathcal T$.

\begin{definition}[\rm Up-down Markov chains]\label{2.3}
Let $\mathcal M$ be a Markov measure on $\mathcal T$ and $\{M^{(n)}\}$ be the
corresponding family of distributions on the sets $\Y_n$. To simplify the
discussion assume that $M^{(n)}(\la)>0$ for all $n$ and all $\la\in\Y_n$, so
that the transition probabilities $p^\up=\{p^\up(\la,\nu)\}$ and
$p^\down=\{p^\down(\la,\mu\}$ are well defined for all edges of the Young
graph.

For each $n=1,2,\dots$, we define a Markov chain with the state space $\Y_n$ in
the following way. Given a diagram $\la\in\Y_n$ we apply first the up
transition probabilities $p^\up(\la,\nu)$ and get a random diagram
$\nu\in\Y_{n+1}$. Then we come back to $\Y_n$ by using the down transition
probabilities $p^\down(\nu,\kappa)$. The composition $\la\to\nu\to\kappa$
constitutes a single step of the chain.

In other words, for two diagrams $\la,\kappa\in\Y_n$, the probability of the
one-step transition $\la\to\kappa$ is equal to
\begin{equation}\label{tag2.4}
\Prob\{\la\to\kappa\}=\sum_{\nu:\,
\la\nearrow\nu\searrow\kappa}p^\up(\la,\nu)p^\down(\nu,\kappa).
\end{equation}

We call this chain the {\it up-down Markov chain\/} of level $n$ associated
with the two systems $p^\down$ and $p^\up$ of transition probabilities
\end{definition}

Likewise, one could introduce the down-up chains using the superposition in the
inverse order, $p^\down\circ p^\up$, but we will not use them.

\begin{definition}[\rm The graph $\wt\Y_n$]\label{2.4}
For any $n=1,2,\dots$, introduce
the graph $\wt\Y_n$ whose vertices are diagrams $\la\in\Y_n$ and whose edges
are couples of distinct diagrams $\la,\kappa\in\Y_n$ for which there exists
$\nu\in\Y_{n+1}$ such that $\la\nearrow\nu\searrow\kappa$. The latter condition
is equivalent to saying that $\kappa$ can be obtained from $\la$ by displacing
a single box to a new position. Note that this is a minimal possible
transformation of a Young diagram preserving the number of boxes. One more
equivalent formulation is as follows: Two diagrams $\la$ and $\kappa$ form an
edge in the graph $\wt\Y_n$ if their symmetric difference $\la\triangle\kappa$
consists of precisely two boxes.
\end{definition}

The up-down chain of level $n$ may be viewed as a nearest neighbor random walk
on the graph $\wt\Y_n$.

\begin{proposition}\label{2.5} For any $n=1,2,\dots$, the up-down Markov chain on
$\Y_n$ determined by \eqref{tag2.4} has a unique stationary distribution, which
is the measure $M^{(n)}$. Moreover, $M^{(n)}$ is the symmetrizing measure, so
that the chain is reversible in the stationary regime.
\end{proposition}

\begin{proof} The fact that $M^{(n)}$ is a stationary distribution follows from
the recurrence relations \eqref{tag2.1} and \eqref{tag2.2}. Indeed,
\eqref{tag2.1} shows that the transition $\la\to\nu$ transforms $M^{(n)}$ to
$M^{(n+1)}$, and \eqref{tag2.2} shows that $\nu\to\kappa$ returns $M^{(n+1)}$
back to $M^{(n)}$.

It is easy to check that the graph $\wt\Y_n$ is connected so that all the
states of the chain are communicating. This proves the uniqueness statement.

The last statement means that
$$
M^{(n)}(\la)\Prob\{\la\to\kappa\}=M^{(n)}(\kappa)\Prob\{\kappa\to\la\}
$$
for any edge $\{\la,\kappa\}$ of the graph $\wt\Y_n$. By virtue of
\eqref{tag2.4}, this can be written as
\begin{equation}\label{tag2.5}
\sum_{\nu}M^{(n)}(\la)p^\up(\la,\nu)p^\down(\nu,\kappa)
=\sum_{\nu}M^{(n)}(\kappa)p^\up(\kappa,\nu)p^\down(\nu,\la).
\end{equation}
Observe that
\begin{equation}\label{tag2.6}
M^{(n)}(\la)p^\up(\la,\nu)=M^{(n+1)}(\nu)p^\down(\nu,\la).
\end{equation}
Indeed, by the very definition of the up and down probabilities,  the both
sides of \eqref{tag2.6} are equal to
$$
\Prob\{\la(n)=\la, \;\la(n+1)=\nu\}.
$$
Now \eqref{tag2.6} and the similar equality with $\la$ replaced by $\kappa$
imply that the both sides of \eqref{tag2.5} are equal to
$$
\sum_{\nu}M^{(n+1)}(\nu)p^\down(\nu,\la)p^\down(\nu,\kappa).
$$

\end{proof}

\section{Down transition probabilities in the Young graph with Jack edge
multiplicities}\label{3}

In this section we introduce a special system $p^\down_\tth$ of down transition
probabilities, which are associated with the Jack symmetric functions. Here
$\tth>0$ is the ``Jack parameter'',  an arbitrary positive number.

Let us start with the particular case $\tth=1$, when the down probabilities
have a simple representation-theoretic meaning. Recall that the diagrams
$\la\in\Y_n$ parametrize the irreducible representations of the group $S_n$.
Let $\dim\la$ stand for the dimension of the corresponding representation of
$S_n$. Then
\begin{equation}\label{tag3.1}
p^\down_1(\la,\mu)=\frac{\dim\mu}{\dim\la}\,, \qquad \mu\nearrow\la.
\end{equation}
The classic {\it Young rule\/} says that the restriction of the irreducible
representation indexed by $\la\in\Y_n$ to the subgroup $S_{n-1}$ splits into
the multiplicity free direct sum of the irreducible representations indexed by
the diagrams $\mu\nearrow\la$. Therefore, for any $\la$,
$$
\sum_{\mu:\,\mu\nearrow\la}\dim\mu=\dim\la,
$$
which explains why the numbers \eqref{tag3.1} sum to 1.

Thus, one can say that the probabilities \eqref{tag3.1} reflect the {\it
branching rule of irreducible representations\/} of symmetric groups.

We proceed to the definition of the down probabilities for general $\tth>0$. We
will present two equivalent formulations. The first one is stated in terms of
the {\it Jack symmetric functions\/}; it  explains the origin of the
probabilities in question. The second one has the advantage of being completely
elementary and will be used in the computations.

For more detail about the notions that will be used below, see \cite{Ma}.

Let $\La$ denote the algebra of symmetric functions over $\R$. It is isomorphic
to the algebra of polynomials with countably many variables $p_1,p_2,\dots$
which are identified with the Newton power sums. The canonical grading of the
algebra $\La$ is specified by setting $\deg p_i=i$. The $n$th homogeneous
component of the algebra, denoted as $\La_n$, has dimension equal to $|\Y_n|$.

All natural bases in $\La$ are indexed by partitions. Of particular importance
for us is the basis $\{\mathcal P_{\mu;\tth}\}_{\mu\in\Y}$ of the Jack
symmetric functions. These are homogeneous elements, the degree of $\mathcal
P_{\mu;\tth}$ equals $|\mu|$.  Recall that Macdonald \cite{Ma} uses as the
parameter the inverse quantity $\tth^{-1}$.

The starting point of the definition is the simplest case of the {\it Pieri
rule\/}: for any $\mu\in\Y$
$$
\mathcal P_{\mu;\tth}\cdot p_1
=\sum_{\la:\,\la\searrow\mu}\varkappa_{\,\tth}(\mu,\la)\mathcal P_{\la;\tth},
$$
where $\varkappa_{\,\tth}(\mu,\la)$ are certain strictly positive numbers
called the {\it Jack formal edge multiplicities\/}. A {\it standard Young
tableau\/} of shape $\la$ is a finite monotone path in the Young graph,
$\varnothing\nearrow\dots\nearrow\la$, starting at $\varnothing$ and ending at
$\la$; its {\it weight\/} is defined as the product of the formal
multiplicities $\varkappa_{\,\tth}(\,\cdot\,,\,\cdot\,)$ of its edges. The {\it
$\tth$-dimension\/} $\dim_\tth\la$ of a diagram $\la$ is defined as the sum of
the weights of all standard tableaux of the shape $\la$. Now we are in a
position to state the definition:

\begin{definition}\label{3.1}
For $\mu\nearrow\la$ we set
\begin{equation}\label{tag3.2}
p^\down_\tth(\la,\mu)=\frac{\dim_\tth\mu\cdot\varkappa_{\,\tth}(\mu,\la)}
{\dim_\tth\la}\,.
\end{equation}
In words: Consider the finite set of all directed paths from $\varnothing$ to
$\la$ and make it a probability space by assigning to each path the probability
proportional to its weight; then $p^\down_\tth(\la,\mu)$ is the probability
that the random path passes through $\mu$. Note that \eqref{tag3.2} is the same
as \eqref{tag1.13}.
\end{definition}

\begin{remark}\label{3.2}
The following duality relation holds :
$$
p^\down_\tth(\la,\mu)=p^\down_{\tth^{-1}}(\la',\mu'), \qquad 0<\tth<+\infty,
$$
where $\la'$ and $\mu'$ are the transposed diagrams.
\end{remark}

\begin{remark}\label{3.3} The construction of the present section first appeared in the
joint paper \cite{KOO}. However, the idea is implicitly contained in an earlier
work by Kerov (see \cite[\S7]{Ke1}).
\end{remark}

The alternative definition of the down transition probabilities is given in the
next section.

\section{Kerov interlacing coordinates and the second definition of the down
transition probabilities}\label{4}

The present section is essentially an extraction from Kerov's paper \cite{Ke4},
with minor modifications.

Let $\la\in\Y$ be a Young diagram. Recall that we draw Young diagrams according
to the so-called ``English picture'' \cite{Ma}, where the first coordinate axis
(the row axis) is directed downwards and the second coordinate axis (the column
axis) is directed to the right. Consider the border line of $\la$ as the
directed path coming from $+\infty$ along the second (horizontal) axis, next
turning several times alternately down and to the left, and finally going away
to $+\infty$ along the first (vertical) axis. The corner points on this path
are of two types: the {\it inner corners\/}, where the path switches from the
horizontal direction to the vertical one, and the {\it outer corners\/} where
the direction is switched from vertical to horizontal. Observe that the inner
and outer corners always interlace and the number of inner corners always
exceeds by 1 that of outer corners. Let $2d-1$ be the total number of the
corners and $(r_i,s_i)$, $1\le i\le 2d-1$, be their coordinates. Here the odd
and even indices $i$ refer to the inner and outer corners, respectively.

\medskip

\begin{center}
 \scalebox{0.5}{\includegraphics{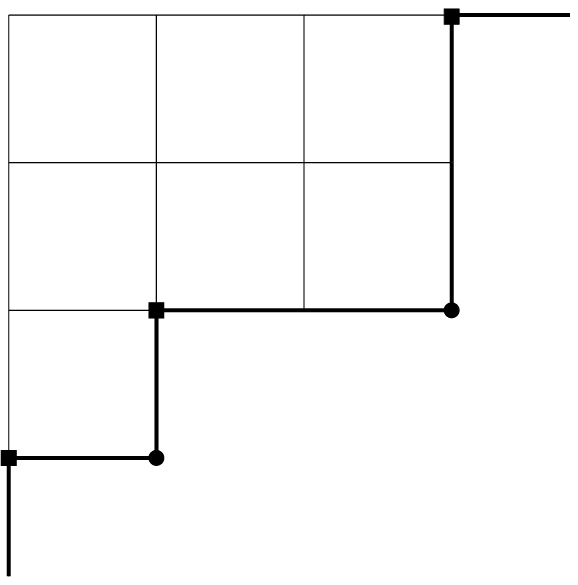}}

Figure 1. The corners of the diagram $\la=(3,3,1)$.
\end{center}

For instance, the diagram $\la=(3,3,1)$ shown on the figure has $d=3$, three
inner corners $(r_1,s_1)=(0,3)$, $(r_3,s_3)=(2,1)$, $(r_5,s_5)=(3,0)$, and two
outer corners $(r_2,s_2)=(2,3)$, $(r_4,s_4)=(3,1)$.

Fix $\tth>0$. The numbers
\begin{multline}\label{tag4.1}
 x_1:=s_1-\tth  r_1, \quad y_1:=s_2-\tth r_2,\,\dots\\
\dots,\,y_{d-1}:=s_{2d-2}-\tth r_{2d-2}, \quad x_d:=s_{2d-1}-\tth r_{2d-1}
\end{multline}
form two interlacing sequences of integers
$$
x_1>y_1>x_2>\dots>y_{d-1}>x_d
$$
satisfying the relation
\begin{equation}\label{tag4.2}
\sum_{i=1}^d x_i-\sum_{j=1}^{d-1}y_j=0.
\end{equation}

\begin{definition}\label{4.1}
The two interlacing sequences
$$
X=(x_1,\dots,x_d), \qquad Y=(y_1,\dots,y_{2d-1})
$$
as defined in \eqref{tag4.1} will be called the ($\tth$-dependent) {\it Kerov
interlacing coordinates\/} of a Young diagram $\la$. We will write $\la=(X;Y)$.
(Note that the original definition of the interlacing coordinates given in
\cite{Ke4} differs from the present one by a factor of $\tth^{-1}$, because
Kerov uses the homothetic transformation $s\mapsto\al s$ with $\al=\tth^{-1}$,
while we prefer to transform the $r$-axis. This minor difference is
inessential: all formulas in \cite{Ke4} can be easily rewritten in our
notation. The term ``anisotropic diagram'' employed in \cite{Ke4} refers to the
image of a Young diagram under a homothetic transformation of a coordinate
axis.)\footnote{In the particular case $\tth=1$, the interlacing coordinates
$(X;Y)$ were introduced in earlier Kerov's paper \cite{Ke2} and further
exploited in \cite{Ke5}; see also \cite{Ke6}. A somewhat similar
parametrization was suggested by Stanley \cite{Sta} and then employed in a
number of recent publications. Stanley's $(p;q)$ coordinates differ from the
Kerov ($\tth=1$) coordinates by a simple linear transformation. }
\end{definition}

Let $u$ be a complex variable and consider the following expansion in partial
fractions
\begin{equation}\label{tag4.3}
\dfrac{\prod\limits_{i=1}^d(u-x_i)}{\prod\limits_{j=1}^{d-1}(u-y_j)}
=u-\sum_{j=1}^{d-1} \frac{\pi^\down_j}{u-y_j}\,.
\end{equation}
Note that the constant term in the right-hand side vanishes because of
\eqref{tag4.2}. The coefficients $\pi^\down_j=\pi^\down_j(X;Y)$ are given by
the formula
\begin{equation}\label{tag4.4}
\pi^\down_j=\pi^\down_j(X;Y)=\,-\,\dfrac{\prod\limits_{i=1}^d(y_j-x_i)}
{\prod\limits_{\substack{1\le l\le d-1\\
l\ne j}}(y_j-y_l)}\,.
\end{equation}
They are strictly positive, and their sum is equal to the area of the shape
$\la$ in the modified coordinates $r'=\tth r$, $s'=s$:
\begin{equation}\label{tag4.5}
\sum_{j=1}^{d-1} \pi^\down_j=\tth|\la|=\Area(X,Y):=\prod_{1\le i\le j\le
d-1}(x_i-y_i)(y_j-x_{j+1}).
\end{equation}

Observe that there is a natural bijective correspondence between the outer
corners of $\la=(X;Y)$ and those boxes that may be removed from $\la$. Thus, we
may associate these boxes with the coordinates $y_j$.

\begin{proposition}\label{4.2}
Let $\la$ be a Young diagram, $(X;Y)$ be its $\tth$-dependent Kerov interlacing
coordinates, $\pi^\down_j=\pi^\down_j(X;Y)$ be the coefficients from\/
\eqref{tag4.3}, given by\/ \eqref{tag4.4}, and\/ $\Area(X;Y)$ be the quantity
defined in\/ \eqref{tag4.5}. Let\/ $\square_j$ denote the corner box in $\la$
associated with the $j$th coordinate $y_j$ in $Y$. Then the $\tth$-dependent
down transition probabilities as defined in \eqref{tag3.2} are given by the
following elementary expression
\begin{equation}\label{tag4.6}
p^\down_\tth(\la,\la\setminus\square_j)=\frac{\pi^\down_j}{\Area(X;Y)}\,.
\end{equation}
\end{proposition}

\begin{proof} See \cite[Section 7]{Ke4}. (Here and below I do not give more
precise references to claims in \cite{Ke4} because they are numbered
differently in the journal version of the paper and its preprint version posted
on arXiv.)
\end{proof}

Note that the right-hand side of \eqref{tag4.6} is a rational fraction in the
Kerov coordinates, and this function {\it does not depend on\/} $\tth$; the
dependence on $\tth$ is hidden in the Kerov coordinates themselves.

Thus, we obtain an alternative description of the down transition
probabilities.

\begin{remark}\label{4.3}
There is an obvious relation between the set of Kerov coordinates with
parameter $\tth$ of a diagram and the set of Kerov coordinates with reversed
parameter $\tth^{-1}$ of the transposed diagram. Specifically, if the former
set is $(X;Y)=\{x_i\}\cup\{y_j\}$ then the latter set is
$(-\tth^{-1}X;-\tth^{-1}Y)=\{-\tth^{-1}x_i\}\cup\{-\tth^{-1}y_j\}$, with the
reversed enumeration. This fact provides a simple proof of the duality stated
in Remark \ref{3.2}.
\end{remark}

\section{The up transition probabilities of the z-measures}\label{5}

There is a host of Markov measures on $\mathcal T$ consistent with the down
probability system $p_\tth^\down$. In this section we exhibit a distinguished
family of Markov measures which depend on $\tth$ and some additional parameters
$z$ and $z'$. We do this by specifying the corresponding up transition
probabilities $p^\up_{\tth,z,z'}$.

Let $\la\in\Y$ be a Young diagram and $(X;Y)$ be its Kerov interlacing
coordinates as defined in \eqref{tag4.1}. Reverse the expression in the
left-hand side of \eqref{tag4.3} and expand it again in partial fractions:
\begin{equation}\label{tag5.1}
\frac{\prod\limits_{j=1}^{d-1}(u-y_j)}{\prod\limits_{i=1}^d(u-x_i)}=\sum_{i=1}
^d\frac { \pi^\up_i } { u-x_i}\,.
\end{equation}
Here the coefficients $\pi^\up_i$ are given by the formula
\begin{equation}\label{tag5.2}
\pi^\up_i=\pi^\up_i(X;Y)=\frac{\prod\limits_{j=1}^{d-1}(x_i-y_j)}
{\prod\limits_{\substack{1\le l\le d\\
l\ne i}}(x_i-x_l)}\,, \qquad i=1,\dots,d.
\end{equation}
Recall that the dependence of the right-hand side on $\tth$ is hidden in the
definition \eqref{tag4.1} of the Kerov coordinates.

Those boxes that may be appended to $\la$ are associated, in a natural way,
with the inner corners of the boundary of $\la$. Consequently, we may also
associate these boxes with the $x$'s: $\square_i\leftrightarrow x_i$.

Assume first that $z$ and $z'$ are arbitrary complex numbers such that
$zz'+\tth n\ne0$ for all $n=0,1,2,\dots$, and set
\begin{equation}\label{tag5.3}
p^\up_{\tth,z,z'}(\la,\la\cup\square_i)=\frac{(z+x_i)(z'+x_i)}{zz'+\tth
n}\cdot\pi^\up_i\,, \qquad n=|\la|,
\end{equation}
where the coefficients $\pi^\up_i$ are the same as in \eqref{tag5.1},
\eqref{tag5.2}.

\begin{proposition}\label{5.1} For any fixed $\la\in\Y_n$, these numbers sum up
to\/ $1$.
\end{proposition}

\begin{proof} Since $(z+x_i)(z'+x_i)=zz'+(z+z')x_i+x_i^2$, the claim is
equivalent to the following three equalities:
$$
\sum_{i=1}^d\pi^\up_i=1, \qquad \sum_{i=1}^d x_i\pi^\up_i=0, \qquad
\sum_{i=1}^d x_i^2\pi^\up_i=\tth n.
$$
These equalities are verified directly from \eqref{tag5.1} using the relation
\eqref{tag4.2} and the expression of $\tth n$ through $(X;Y)$, see
\eqref{tag4.5}. For more detail, see \cite[Section 6]{Ke4}.
\end{proof}

Let us define the numbers $M^{(n)}(\la)=M^{(n)}_{\tth,z,z'}(\la)$ from the
recurrence relations \eqref{tag2.1} by setting $p^\up=p^\up_{\tth,z,z'}$ and
using the initial condition $M^{(0)}(\varnothing)=1$. Then, using Proposition
\ref{5.1} and induction on $n$, one sees that for all $n$,
$$
\sum_{\la\in\Y_n}M^{(n)}_{\tth,z,z'}(\la)=1.
$$

\begin{proposition}\label{5.2} The numbers $M^{(n)}(\la)=M^{(n)}_{\tth,z,z'}(\la)$
just defined  are consistent with the down probabilities
$p^\down_\tth(\la,\mu)$. That is, setting $p^\down=p^\down_\tth$, the relations
\eqref{tag2.2} are satisfied.
\end{proposition}

\begin{proof} This the main result of \cite{Ke4}; it is established
at the very end of that paper. Another proof is given in \cite{BO3}.

\end{proof}

\begin{proposition}\label{5.3} The quantities $p^\up_{\tth,z,z'}(\la,\nu)$ defined
by \eqref{tag5.3} are strictly positive for all edges $\la\nearrow\nu$ of the
Young graph if and only if one of the following two conditions holds\/{\rm:}

{\rm(i)} $z\in\C\setminus (\Z_{\le0}+\tth\cdot\Z_{\ge0})$ and $z'=\bar z$.

{\rm(ii)} $\tth$ is rational and both $z$ and $z'$ are real numbers lying in
one of the open intervals between two consecutive numbers from the lattice
$\Z+\tth\cdot\Z\subset\R$.
\end{proposition}

In particular, the simple sufficient condition of strict positivity is that $z$
and $z'$ should be nonreal and conjugate to each other.

\begin{proof} See \cite[Proposition 2.2]{BO5}.
\end{proof}

\begin{definition}\label{5.4}
We say that the couple $(z,z')$ belongs to the {\it principal series\/} or to
the {\it complementary series\/} if it satisfies {\rm(i)} or {\rm(ii)},
respectively. Of course, the complementary series exists for rational $\tth$
only.
\end{definition}

Let us summarize the results of this and preceding sections:

Let $\tth>0$ and let $(z,z')$ belong to the principal or complementary series.
For each $n=1,2,\dots$ all the numbers $M^{(n)}_{\tth,z,z'}(\la)$,
$\la\in\Y_n$, are strictly positive and sum up to\/ $1$, and hence they
determine a probability measure $M^{(n)}_{\tth,z,z'}$ on $\Y_n$. By the very
construction, these measures are consistent with the up transition
probabilities $p^\up_{\tth,z,z'}(\la,\nu)$ defined in \eqref{tag5.3}. They are
also consistent with the down transition probabilities $p^\down_\tth(\la,\mu)$
defined in \eqref{tag4.6}.

We call the measures $M^{(n)}_{\tth,z,z'}$ the {\it z-measures\/} with Jack
parameter $\tth$.  An explicit expression for the weights
$M^{(n)}_{\tth,z,z'}(\la)$ is given in \cite{BO5}, but in the present paper we
will not need it.

\section{The algebra $A_\tth$ of $\tth$-regular functions on Young
diagrams}\label{6}

In this section we fix an arbitrary $\tth>0$. For a set $\frak X$, we will
denote by $\Fun(\frak X)$ the algebra of all real-valued functions on $\frak
X$. Below $\la$ stands for an arbitrary Young diagram.

Let $u$ be a complex variable. Set
$$
\Phi(u;\la)=\prod_{i=1}^\infty\frac{u+\tth i}{u-\la_i+\tth i}
$$
and observe that the product is actually finite because only finitely many of
$\la_i$'s differ from $0$, so that only finitely many factors differ from $1$.
Clearly, for $\la$ fixed, $\Phi(u;\la)$ is a rational function in $u$ taking
value 1 at $u=\infty$. Therefore, $\Phi(u;\la)$ admits the Taylor expansion at
$u=\infty$ with respect to the variable $u^{-1}$. Likewise, such an expansion
also exists for $\log \Phi(u;\la)$.

\begin{definition}\label{6.1}
Let $\A_\tth\subset \Fun(\Y)$ be the unital subalgebra
generated by the coefficients of the Taylor expansion at $u=\infty$ of
$\Phi(u;\la)$ (or, equivalently, of $\log\Phi(u;\la)$). We call $\A_\tth$ the
{\it algebra of $\tth$-regular functions on $\Y$}.
\end{definition}

The Taylor expansion of $\log\Phi(u;\la)$ at $u=\infty$ has the form
\begin{equation}\label{tag6.1}
\log\Phi(u;\la)=\sum_{m=1}^\infty\frac{p^*_m(\la)}m\,u^{-m},
\end{equation}
where, by definition,
$$
p^*_m(\la)=\sum_{i=1}^\infty [(\la_i-\tth i)^m-(-\tth i)^m], \qquad
m=1,2,\dots, \quad \la\in\Y.
$$
The above expression makes sense because the sum is actually finite. Thus, the
algebra $\A_\tth$ is generated by the functions $p^*_1,p^*_2,\dots$. It is
readily verified that these functions are algebraically independent, so that
$\A_\tth$ is isomorphic to the algebra of polynomials in the variables
$p^*_1,p^*_2,\dots$.

\begin{definition}\label{6.2}
Using the isomorphism $\A_\tth\cong\R[p^*_1,p^*_2,\dots]$
we define a {\it filtration\/} in $\A_\tth$ by setting $\deg
p^*_m(\,\cdot\,)=m$. In more detail, the $m$th term of the filtration,
consisting of elements of degree $\le m$, is the finite-dimensional subspace
$\A_\tth^{(m)}\subset\A_\tth$ defined in the following way:
$$
\A_\tth^{(0)}=\R1; \quad
\A_\tth^{(m)}=\operatorname{span}\{(p^*_1)^{r_1}(p^*_2)^{r_2}\dots\,:\,
1r_1+2r_2+\dots\le m\},\quad m=1,2,\dots\,.
$$
\end{definition}

Note that $p^*_1(\la)=|\la|$. The $\tth$-regular functions on $\Y$ (that is,
elements of $\A_\tth$) coincide with the $\tth$-shifted symmetric polynomials
in the variables $\la_1,\la_2,\dots$ as defined in \cite{OO}, \cite{KOO}.

Next, we set
\begin{equation}\label{tag6.2}
\HH(u;\la)=\frac{u\prod\limits_{j=1}^{d-1}(u-y_j)}{\prod\limits_{i=1}^d(u-x_i)}\,,
\qquad \HE(u;\la)=\frac{-1}{\HH(u;\la)}
=-\,\frac{\prod\limits_{i=1}^d(u-x_i)}{u\prod\limits_{j=1}^{d-1}(u-y_j)}\,,
\end{equation}
where $x_1,\dots,x_d,y_1,\dots,y_{d-1}$ are the $\tth$-dependent Kerov
interlacing coordinates of $\la$ defined in \eqref{tag4.1}. Consider the Taylor
expansions of $\HH(u;\la)$ and $\HE(u;\la)$ at $u=\infty$:
\begin{equation}\label{tag6.3}
\HH(u;\la)=1+\sum_{m=1}^\infty\h_m(\la)u^{-m}, \qquad
\HE(u;\la)=-1+\sum_{m=1}^\infty\he_m(\la)u^{-m}.
\end{equation}
Because of \eqref{tag4.2} we have
$$
\h_1(\la)=\he_1(\la)\equiv0.
$$

Further, we set
\begin{equation}\label{tag6.4}
\p_m(\la)=\sum_{i=1}^d x_i^m-\sum_{j=1}^{d-1}y_j^m, \qquad m=1,2,\dots.
\end{equation}
Obviously,
\begin{equation}\label{tag6.5}
\log\HH(u;\la)=\sum_{m=1}^\infty \frac{\p_m(\la)}m\,u^{-m}, \qquad
\p_1(\la)\equiv0.
\end{equation}

\begin{proposition}\label{6.3} The following relation holds
\begin{equation}\label{tag6.6}
\HH(u;\la)=\frac{\Phi(u-\tth;\la)}{\Phi(u;\la)}\,.
\end{equation}
\end{proposition}

For $\tth=1$, another proof (due to Kerov) can be found in \cite[Proposition
3.6]{IO}  (note that the definition of $\Phi(u,\la)$ in \cite{IO} differs from
our definition by  a shift of the argument $u$).

\begin{proof} We proceed by induction on $n=|\la|$. For $n=0$ there exists only
one diagram, the empty one. The relation \eqref{tag6.6} is satisfied because
$\Phi(u,\varnothing)=\HH(u;\varnothing)\equiv1$.

Let us examine the transformation of the both sides of \eqref{tag6.6} when one
appends a box $\square=(i,j+1)$ to $\la$.

In terms of the row coordinates, this means that the coordinate $\la_i=j$ is
increased by 1. Consequently,
$$
\frac{\Phi(u;\la\cup\square)}{\Phi(u;\la)}=\frac{u-\la_i+\tth i}{u-\la_i-1+\tth
i}=\frac{u-j+\tth i}{u-j-1+\tth i}\,,
$$
which implies that the right-hand side of \eqref{tag6.6} is multiplied by
$$
\frac{(u-j-1+\tth i)(u-\tth-j+\tth i)}{(u-j+\tth i)(u-\tth-j-1+\tth i)}\,.
$$

On the other hand, recall that there is a natural bijective correspondence
between the boxes that may be appended to $\la=(X;Y)$ and the points in $X$.
Observe that the point $x$ corresponding to the box $\square=(i,j+1)$ is
$j-\tth(i-1)$. Therefore, the above expression can be rewritten as
$$
\frac{(u-x)(u-x+\tth-1)}{(u-x-1)(u-x+\tth)}\,.
$$

Now, the lemma below implies that the left-hand side of \eqref{tag6.6}
undergoes precisely the same transformation. This completes the induction step.
\end{proof}

\begin{lemma}\label{6.4} Let $\la=(X;Y)$ be a Young diagram and $\la\cup\square$ be
the diagram obtained from $\la$ by appending the box\/ $\square$ corresponding
to a point $x\in X$. Then
\begin{equation}\label{tag6.7}
\frac{\HH(u;\la\cup\square)}{\HH(u;\la)}=\frac{(u-x)(u-x+\tth-1)}{(u-x-1)(u-x+\tth)}\,.
\end{equation}
\end{lemma}

\begin{proof} Let $y'$ and $y''$ be the neighboring points to $x$ in $Y$,
$y'>x>y''$. If $x$ is the greatest element of $X$ then $y'$ does not exist, and
if $x$ is the smallest element then $y''$ does not exist (these two extreme
cases occur when $\square$ lies in the first row or in the first column,
respectively). Observe that if $y'$ exists then the difference $y'-x$ is in
$\{1,2,3,\dots\}$, and if $y''$ exists then the difference $x-y''$ is in
$\{\tth,2\tth,3\tth\dots\}$.

Write $\la\cup\square=(\bar X;\bar Y)$. Consider first the generic case, when
$y'-x\ne1$ and $x-y''\ne\tth$. \footnote{If $y'$ does not exist then we
formally set $y'-x=+\infty\ne1$. Likewise, if $y''$ does not exist we formally
set $x-y''=+\infty\ne\tth$.} It is readily checked that
$$
\bar X=(X\setminus\{x\})\cup\{x+1,x-\tth\}, \qquad \bar Y=Y\cup\{x+1-\tth\}.
$$
Then \eqref{tag6.7} follows directly from the definition of $\HH(u;\la)$ in
\eqref{tag6.2}.

Let us examine now the remaining cases:

If $y'-x=1$ while $x-y''\ne\tth$ then
$$
\bar X=(X\setminus\{x\})\cup\{x-\tth\}, \qquad \bar
Y=(Y\setminus\{y'\})\cup\{x+1-\tth\}=(Y\setminus\{x+1\})\cup\{x+1-\tth\}.
$$

If $x-y''=\tth$ while $y'-x\ne1$ then
$$
\bar X=(X\setminus\{x\})\cup\{x+1\}, \qquad \bar
Y=(Y\setminus\{y''\})\cup\{x+1-\tth\}=(Y\setminus\{x-\tth\})\cup\{x+1-\tth\}.
$$

Finally, if both $y'-x=1$ and $x-y''=\tth$, then
$$
\bar X=X\setminus\{x\}, \qquad \bar Y=(Y\setminus\{y',y''\})\cup\{x+1-\tth\}
=(Y\setminus\{x+1,x-\tth\})\cup\{x+1-\tth\}.
$$

Again, in each of these three cases one readily checks that \eqref{tag6.7}
remains true. \end{proof}

\begin{proposition}\label{6.5} The functions\/ $\p_m(\la)$ defined in\/ \eqref{tag6.4}
belong to the algebra $\A_\tth$. More precisely, we have
\begin{equation}\label{tag6.8}
\p_m=\tth\cdot m\cdot p^*_{m-1}\,+\,\dots, \qquad m=2,3,\dots,
\end{equation}
where dots stand for lower degree terms, which are a linear combination of
elements $p^*_l$, with $1\le l\le m-2$.
\end{proposition}

\begin{proof} Combining \eqref{tag6.1}, \eqref{tag6.5}, and \eqref{tag6.6} we get
\begin{multline*}
 \sum_{m=1}^\infty
\frac{\p_m(\la)}{m}\,u^{-m}=\sum_{l=1}^\infty\frac{p^*_l(\la)}{l}
\left((u-\tth)^{-l}-u^{-l}\right)\\
=\sum_{l=1}^\infty\frac{p^*_l(\la)}{l}\,u^{-l} \left((1-\tth\cdot
u^{-1})^{-l}-1\right)\\ =\sum_{l=1}^\infty\frac{p^*_l(\la)}{l}\,u^{-l}
\left(\tth\cdot l\cdot u^{-1}-\tth^2\cdot\frac{l(l+1)}{2}\,u^{-2}+
\tth^3\cdot\frac{l(l+1)(l+2)}{2\cdot3}u^{-3}-\dots\right)\\
=\sum_{l=1}^\infty p^*_l(\la)\left(\tth\cdot
u^{-(l+1)}-\tth^2\cdot\frac{l+1}{2}\,u^{-(l+2)}+
\tth^3\cdot\frac{(l+1)(l+2)}{2\cdot3}u^{-(l+3)}-\dots\right).
\end{multline*}
Equating the coefficients we obtain the desired claim. More precisely:
$$
\frac{\p_m}m=\tth\cdot p^*_{m-1}-\tth^2\cdot\frac{l+1}2\,p^*_{m-2}
+\tth^3\cdot\frac{(l+1)(l+2)}{2\cdot3}\,p^*_{m-3}-\dots
$$

\end{proof}

\begin{corollary}\label{6.6} Each of the three families of functions
$\{\p_2,\p_3,\dots\}$, $\{\h_2,\h_3,\dots\}$, $\{\he_2,\he_3,\dots\}$ is a
system of algebraically independent generators of the algebra $\A_\tth$.
\end{corollary}

\begin{corollary}\label{6.7} Under the identification of $\A_\tth$ with any of the
three algebras of polynomials
$$
\R[\p_2,\p_3,\dots], \qquad \R[\h_2,\h_3,\dots], \qquad \R[\he_2,\he_3,\dots]
$$
the filtration introduced in Definition \ref{6.2} is determined by setting
$$
\deg \p_m=m-1, \qquad \deg \h_m=m-1, \qquad \deg \he_m=m-1 \qquad
(m=2,3,\dots),
$$
respectively.
\end{corollary}

Let $\La$ be the algebra of symmetric functions over the base field $\R$.
Following Macdonald \cite{Ma}, we will denote by $\{p_1,p_2,\dots\}$ and
$\{h_1,h_2,\dots\}$ the two systems of generators consisting of the Newton
power sums and the complete homogeneous symmetric functions.

\begin{definition}\label{6.8}
Define the {\it covering homomorphism\/} $\La\to\A_\tth$ by the specialization
\begin{equation}\label{tag6.9}
p_1\to0, \quad p_2\to\p_2, \quad  p_3\to\p_3, \quad \dots
\end{equation}
\end{definition}

By virtue of Corollary \ref{6.6}, the covering homomorphism is surjective and
its kernel is the principal ideal generated by $p_1$. Note also that under
\eqref{tag6.9} we have
\begin{equation}\label{tag6.10}
h_1\to0, \quad h_2\to\h_2, \quad  h_3\to\h_3, \quad \dots
\end{equation}

\begin{corollary}\label{6.9} If $F(\la)$ is a function from $\A_\tth$ then the
function $\la\mapsto F(\la')$ is in $\A_{\tth^{-1}}$.
\end{corollary}

\begin{proof} To make explicit the dependence on $\tth$, introduce the more
detailed notation $\p_{\tth,m}$ instead of $\p_m$. Using this notation and
Remark \ref{4.3}, we have
$$
\p_{\tth^{-1},m}(\la')=(-\tth)^{-m}\p_{\tth,m}(\la),
$$
which implies the claim.
\end{proof}

The next two lemmas will be used in Section \ref{7} below.

\begin{lemma}\label{6.10} Let $\la=(X;Y)$ be a Young diagram and
$\la\setminus\square$ be the diagram obtained from $\la$ by removing the box\/
$\square$ corresponding to a point $y\in Y$. Then
\begin{equation}\label{tag6.11}
\frac{\HH(u;\la\setminus\square)}{\HH(u;\la)}=\frac{(u-y+1)(u-y-\tth)}{(u-y)(u-y-\tth+1)}\,.
\end{equation}
\end{lemma}

\begin{proof}
The argument is similar to that in the proof of Lemma \ref{6.4}. Write
$\la\setminus\square=(\bar X;\bar Y)$. Let $x'$ and $x''$ be the neighboring
points to $y$ in $X$, $x'>y>x''$. In the generic case, when $x'>y+\tth$ and
$x''<y-1$, we have
$$
\bar X=X\cup\{y+\tth-1\},\qquad \bar Y=(Y\setminus\{y\})\cup\{y+\tth,y-1\},
$$
and the claim follows from the definition of $\HH(u;\la)$. The remaining
possible cases are examined as in the proof of Lemma \ref{6.4}.
\end{proof}

\begin{lemma}\label{6.11} Let $\la$ be a Young diagram; $X=\{x_1,\dots,x_d\}$ and\/
$Y=\{y_1,\dots,y_{d-1}\}$ be its Kerov interlacing coordinates;
$\pi^\up_1,\dots,\pi^\up_d$ be the numbers associated to $(X;Y)$ according to
formulas \eqref{tag5.1} and \eqref{tag5.2}; $\pi^\down_1,\dots,\pi^\down_{d-1}$
be the numbers associated to $(X;Y)$ according to formulas \eqref{tag4.3} and
\eqref{tag4.4}.

Then for $m=0,1,2,\dots$
\begin{equation}\label{tag6.12}
\sum_{i=1}^d \pi^\up_i x_i^m=\h_m(\la), \qquad \sum_{j=1}^{d-1}
\pi^\down_jy_j^m=\he_{m+2}(\la).
\end{equation}
\end{lemma}

\begin{proof} Comparing the definition of $\HH(u;\la)$ (see \eqref{tag6.2}) with
\eqref{tag5.1}, we get
$$
\HH(u;\la)=u\sum_{i=1}^d\frac{\pi_i^\up}{u-x_i}
=\sum_{i=1}^d\pi_i^\up\left(1+\frac{x_i}{u}+\frac{x_i^2}{u^2}+\dots\right).
$$
Equating the coefficients in $u^{-m}$ gives the first equality in
\eqref{tag6.12}.

Likewise, from the definition of $\HE(u;,\la)$ (see \eqref{tag6.2}) and
\eqref{tag4.3} it follows
$$
\HE(u;\la)=-1+\frac1{u}\sum_{j=1}^{d-1}\frac{\pi_j^\down}{u-y_j}
=-1+\sum_{j=1}^{d-1}\pi_j^\down\left(\frac1{u^2}+\frac{y_j}{u^3}+\frac{y_j^2}{u^4}+\dots\right),
$$
which implies the second equality in \eqref{tag6.12}.
\end{proof}

Finally, consider one more set of generators in $\La$, the elementary symmetric
functions $e_1,e_2,\dots$.

\begin{lemma}\label{6.12} Let $\e_1,\e_2,\dots$ denote the images of
$e_1,e_2,\dots$ under the covering homomorphism\/ $\La\to\A_\tth$, see\/
\eqref{tag6.9} and\/ \eqref{tag6.10}.

We have\/ $\he_1=0$ and\/ $\he_m=(-1)^{m-1}\e_m$ for $m=2,3,\dots$\,.
\end{lemma}

\begin{proof} Consider the generating series
$$
H(u)=1+\sum_{m=1}^\infty h_m u^{-m}, \qquad E(u)=1+\sum_{m=1}^\infty e_m
u^{-m}.
$$
By the very definition, the covering homomorphism send $H(u)$ to
$\HH(u;\,\cdot\,)$. On the other hand, $E(u)=1/H(-u)$, so that the covering
homomorphism send $E(u)$ to $\bold E(u;\,\cdot\,):=1/\HH(-u;\,\cdot\,)$.
Comparing this with the definition $\HE(u;\,\cdot\,)=-1/\HH(u;\,\cdot\,)$ (see
\eqref{tag6.2}) we conclude that $\HE(u;\,\cdot\,)=-\bold E(-u;\,\cdot\,)$.
This implies the claim.
\end{proof}

\section{The operators $D$ and $U$ in the algebra $\A_\tth$}\label{7}

In this section we fix a triple of parameters $(\tth,z,z')$. We assume
$\tth>0$. The parameters $z$ and $z'$ may be arbitrary complex numbers such
that right-hand side of formula \eqref{tag5.3} makes sense, so that the numbers
$p^\up_{\tth,z,z'}(\la,\nu)$ are well defined. Since we will be dealing with
formal computations we will not need to require these numbers to be positive.
Thus, in this section, the only restriction on $(z,z')$ is that $zz'+\tth
n\ne0$ for all $n=0,1,2,\dots$\,.

\subsection{Statement of the result}\label{7-1}
Let $D_{n+1,n}: \Fun(\Y_n)\to \Fun(\Y_{n+1})$ and $U_{n,n+1}: \Fun(\Y_{n+1})\to
\Fun(\Y_n)$ be the ``down'' and ``up'' operators acting on functions:
\begin{gather*}
 (D_{n+1,n}F)(\nu)=\sum_{\la\in\Y_n}p^\down_\tth(\nu,\la)F(\la),
\qquad F\in \Fun(\Y_n), \quad
\nu\in\Y_{n+1}\,,\\
(U_{n,n+1}G)(\la)=\sum_{\nu\in\Y_{n+1}}p_{\tth,z,z'}^\up(\la,\nu)G(\nu), \qquad
G\in \Fun(\Y_{n+1}), \quad\la\in\Y_n\,.
\end{gather*}
This action arises by dualizing the natural action of $p^\down_\tth$ and
$p_{\tth,z,z'}^\up$ on measures, which explains the seeming contradiction: the
``down'' operator raises the level $n$ while the ``up'' operator reduces the
level.

In the formulation of Theorem \ref{7.1} below we identify $\A_\tth$ with the
polynomial algebra $\R[\h_2,\h_3,\dots]$. Recall that $\A_\tth$ is a filtered
algebra (Definition \ref{6.2}) and that under the identification
$\A_\tth=\R[\h_2,\h_3,\dots]$ the filtration is determined by setting $\deg
\h_m=m-1$ (Corollary \ref{6.7}. We say that an operator $\mathcal A:
\A_\tth\to\A_\tth$ has degree $\le m$, where $m\in\Z$, if for any
$F\in\A_\tth$, $\deg(\mathcal A F)\le \deg F+m$.

Observe that any operator in the algebra of polynomials (in finitely or
countably many variables) can be written as a differential operator with
polynomial coefficients, that is, as a formal infinite sum of differential
monomials. This fact is well known and can be readily proved; we do not use it
but it is helpful to take it in mind while reading the formulation and the
proof of Theorem \ref{7.1}.

Given $F\in\A_\tth$, we denote by $F_n$ the restriction of the function
$F(\,\cdot\,)$ to $\Y_n\subset\Y$. It is readily checked that the subalgebra
$\A_\tth\subset\Fun(\Y)$ separates points, which implies that for each $n$, the
functions of the form $F_n$, with $F\in\A_\tth$, exhaust the space
$\Fun(\Y_n)$.

\begin{theorem}\label{7.1}
{\rm(i)} There exists a unique operator $D:\A_\tth\to\A_\tth$  such that
$$
D_{n+1,n}F_n=\frac1{\tth(n+1)}(DF)_{n+1}\,, \qquad\text{for all $n=0,1,\dots$
and all $F\in\A_\tth$}.
$$
More precisely, the operator $D$ has degree $1$ with respect to the filtration
of $\A_\tth$, and its top degree terms look as follows
\begin{gather*}
 D=\h_2+\frac12\tth^2\sum_{r,s\ge2}(r-1)(s-1)
\h_{r+s-2}\frac{\pd^2}{\pd \h_r\pd\h_s}\\
-\tth\sum_{r\ge2}(r-1)\h_r\frac{\pd}{\pd\h_r}\\
+\frac12\tth(1-\tth)\sum_{r\ge3}(r-1)(r-2)\h_{r-1}\frac{\pd}{\pd\h_r}\\
+\frac12\tth\sum_{r,s\ge2}(r+s)\h_r\h_s\frac{\pd}{\pd\h_{r+s}}\\
+\text{\rm terms of degree $\le-2$}.
\end{gather*}

{\rm(ii)} There exists a unique operator $U:\A_\tth\to\A_\tth$  such that
$$
U_{n,n+1}F_{n+1}=\frac1{zz'+\tth n}(UF)_n\,, \qquad\text{for all $n=0,1,\dots$
and all $F\in\A_\tth$}.
$$
More precisely, the operator $U$ has degree $1$ with respect to the filtration
of $\A_\tth$, and its top degree terms look as follows
\begin{gather*}
U=\h_2+zz'+\tth zz'\frac{\pd}{\pd\h_2}\\
+\tth(z+z')\sum_{r\ge3}(r-1)\h_{r-1}\frac{\pd}{\pd\h_r}\\
+\frac12\tth^2\sum_{r,s\ge2}(r-1)(s-1)
\h_{r+s-2}\frac{\pd^2}{\pd \h_r\pd\h_s}\\
+\tth\sum_{r\ge2}(r-1)\h_r\frac{\pd}{\pd\h_r}\\
+\frac12\tth(1-\tth)\sum_{r\ge3}(r-1)(r-2)\h_{r-1}\frac{\pd}{\pd\h_r}\\
+\frac12\tth\sum_{r,s\ge2}(r+s-2)\h_r\h_s\frac{\pd}{\pd\h_{r+s}}\\
+\text{\rm terms of degree $\le-2$}.
\end{gather*}
\end{theorem}

Note that $D$ depends only on $\tth$ while $U$ depends on the whole triple
$(\tth,z,z')$.

The rest of the section is devoted to the proof. Since it is long, let us
briefly describe its idea. Instead of dealing with individual elements of
$\A_\tth$ it is more convenient to manipulate with generating series. We know
that the series $\HH(u;\la)$ gathers the generators $\h_2(\la),
\h_3(\la),\dots$ of the algebra $\A_\tth$. Therefore, the products
$\HH(u_1;\la)\HH(u_2;\la)\dots$ gather various products of the generators,
which in turn constitute a linear basis in $\A_\tth$. Thus, we know the action
of our operators if we know how they transform products of generating series.
Now, it turns out that the transformation of $\HH(u_1;\la)\HH(u_2;\la)\dots$
can be written down in a closed form. {}From this we can extract all the
necessary information.

\subsection{Action of $D$ and $U$ on products of generating series}\label{7-2}
We proceed to the detailed proof. Recall that
\begin{equation}\label{tag7.1}
H(u)=1+\sum_{m=0}^\infty h_m u^{-m}\,\in\,\La[[u^{-1}]].
\end{equation}
Let $\rho=(\rho_1,\rho_2,\dots)$ range over the set of partitions. Recall the
standard notation $h_\rho$ and $m_\rho$ for the complete homogeneous symmetric
functions and monomial symmetric functions, respectively, see \cite{Ma}. Take a
finite collection of variables $u_1,u_2,\dots$ (we prefer to not indicate their
number explicitly). Then
$$
\prod_l H(u_l)=\sum_{\rho}m_\rho(u_1^{-1},u_2^{-1},\dots)h_\rho
$$
summed over all $\rho$'s such that $\ell(\rho)$ (the number of nonzero parts in
$\rho$) does not exceed the number of variables $u_1,u_2,\dots$\,.

Applying to \eqref{tag7.1} the covering  homomorphism $\La\to\A_\tth$
(Definition \ref{6.8}) and using \eqref{tag6.10} we get
\begin{equation}\label{tag7.2}
\prod_l \HH(u_l;\la)=\sum_{\rho}m_\rho(u_1^{-1},u_2^{-1},\dots)\h_\rho(\la).
\end{equation}
Because $\h_1(\la)\equiv0$, we may and do additionally assume that $\rho$ does
not contain parts equal to 1. Note that the set $\{\h_\rho:
\rho_1,\rho_2,\ldots\ne1\}$ is a basis in $\A_\tth$.

We regard the left-hand side of \eqref{tag7.2} as a generating series for the
basis elements $\h_\rho$. Thus, the transformation of the left-hand side under
the action of an operator acting on functions in $\la$ is completely determined
by its action on the functions $\h_\rho(\la)$ in the right-hand side.

Occasionally, it will be convenient to omit the argument $\la$ in
the notation $\HH(u;\la)$. Recall also the notation $(\dots)_n$ for
the operation of restriction to the subset $\Y_n\subset\Y$.

By the very definition of $U_{n,n+1}$ we have
\begin{multline*}
\left(U_{n,n+1}\left(\prod_l\HH(u_l)\right)_{n+1}\right)(\la)\\=\sum_{i=1}^d
p^\up_{\tth,z,z'}(\la,\la\cup\square_i) \prod_l\HH(u_l,\la\cup\square_i),
\qquad \la\in\Y_n
\end{multline*}
Substituting the explicit expression \eqref{tag5.3} for the up probabilities
and using Lemma \ref{6.4} we rewrite this equality as
\begin{multline}\label{tag7.3}
\left((zz'+\tth
n)U_{n,n+1}\left(\prod_l\HH(u_l)\right)_{n+1}\right)(\la)\\=\left\{\sum_{i=1}^d
(z+x_i)(z'+x_i)\prod_l\frac{(u_l-x_i)(u_l-x_i+\tth-1)}{(u_l-x_i-1)(u_l-x_i+\tth)}
\cdot\pi^\up_i(X;Y)\right\}\\
\times \prod_l\HH(u_l;\la), \qquad \la\in\Y_n.
\end{multline}
Here, as usual,  $(X;Y)$ are the Kerov interlacing coordinates of $\la$ and
$\pi^\up_i(X;Y)$ denote the numbers $\pi^\up_i$ defined in \eqref{tag5.2}.

Likewise, by the definition of $D_{n+1,n}$,
\begin{multline*}
\left(D_{n+1,n}\left(\prod_l\HH(u_l)\right)_n\right)(\la)\\
=\sum_{j=1}^{d-1} p^\down(\la,\la\setminus\square_j)
\prod_l\HH(u_l,\la\setminus\square_j), \qquad \la\in\Y_{n+1}
\end{multline*}
Substituting the explicit expression for the down probabilities (see
\eqref{tag4.6} and \eqref{tag4.5}) and using Lemma \ref{6.10} we rewrite this
as
\begin{multline}\label{tag7.4}
\left(\tth(n+1)D_{n+1,n}\left(\prod_l\HH(u_l)\right)_n\right)(\la)
\\=\left\{\sum_{j=1}^{d-1}
\prod_l\frac{(u_l-y_j+1)(u_l-y_j-\tth)}{(u_l-y_j)(u_l-y_j-\tth+1)}
\cdot\pi^\down_j(X;Y)\right\}\\
\times \prod_l\HH(u_l;\la), \qquad \la\in\Y_{n+1}.
\end{multline}
Here $\pi^\down_j(X;Y)$ are the numbers $\pi^\down_j$ defined in
\eqref{tag4.4}.

It is convenient to introduce a special notation for the expressions in the
curly brackets that appear in \eqref{tag7.3} and \eqref{tag7.4}:
\begin{gather} F^\up(u_1,u_2,\dots;\la)=\sum_{i=1}^d
(z+x_i)(z'+x_i)\prod_l\frac{(u_l-x_i)(u_l-x_i+\tth-1)}{(u_l-x_i-1)(u_l-x_i+\tth)}
\cdot\pi^\up_i(X;Y)  \label{tag7.5}\\
F^\down(u_1,u_2,\dots;\la)= \sum_{j=1}^{d-1}
\prod_l\frac{(u_l-y_j+1)(u_l-y_j-\tth)}{(u_l-y_j)(u_l-y_j-\tth+1)}
\cdot\pi^\down_j(X;Y)  \label{tag7.6}
\end{gather}

\begin{lemma}\label{7.2} As functions in $\la$, both $F^\up(u_1,u_2,\dots;\la)$ and
$F^\down(u_1,u_2,\dots;\la)$ are elements of the algebra $\A_\tth$. More
precisely, the both expressions can be viewed as elements of
$\A_\tth[[u_1^{-1},u_2^{-1},\dots]]$.
\end{lemma}

\begin{proof} Observe that the $i$th product in \eqref{tag7.5} and the $j$th product
in \eqref{tag7.6} can be viewed as elements of
$\R[x_i][[u_1^{-1},u_2^{-1},\dots]]$ and $\R[y_j][[u_1^{-1},u_2^{-1},\dots]]$,
respectively, and then apply Lemma \ref{6.11}.
\end{proof}

Formulas \eqref{tag7.3} and \eqref{tag7.4} combined with Lemma \ref{7.2} show
that the operators $(zz'+\tth n)U_{n,n+1}$ and $\tth(n+1)D_{n+1,n}$ are indeed
induced by certain operators $U$ and $D$ acting in $\A_\tth$, and the
transformation of the generating series $\HH(u_1)\HH(u_2)\dots$ under the
action of these two operators looks as follows (it is convenient to omit the
argument $\la$ in the formulas below):
\begin{equation}\label{tag7.7}
\begin{aligned}
 U(\HH(u_1)\HH(u_2)\dots)&=F^\up(u_1,u_2,\dots)\HH(u_1)\HH(u_2)\dots\\
D(\HH(u_1)\HH(u_2)\dots)&=F^\down(u_1,u_2,\dots)\HH(u_1)\HH(u_2)\dots
\end{aligned}
\end{equation}

These nice formulas contain in a compressed form all the information about the
action of $U$ and $D$ on the basis elements $\h_\rho$. Our next step is to
extract from \eqref{tag7.7} some explicit expressions for $U\h_\rho$ and
$D\h_\rho$ using \eqref{tag7.2} and Lemma \ref{6.11}.

\subsection{Action of $D$ and $U$ in the basis $\{\h_\rho\}$}\label{7-3}

We need to introduce some notation. Expand the products \eqref{tag6.7} and
\eqref{tag6.11} about $u=\infty$:
\begin{equation}\label{tag7.8}
\begin{aligned}
\frac{(u-x)(u-x+\tth-1)}{(u-x-1)(u-x+\tth)}&=\sum_{s=0}^\infty
a_s(x)u^{-s}, \qquad a_s\in\R[x],\\
\frac{(u-y+1)(u-y-\tth)}{(u-y)(u-y-\tth+1)}&= \sum_{s=0}^\infty b_s(y)u^{-s},
\qquad b_s\in\R[y].
\end{aligned}
\end{equation}

\begin{lemma}\label{7.3} We have
$$
a_0(x)=b_0(y)\equiv1, \qquad a_1(x)=b_1(y)\equiv0,
$$
and $a_s(x)$ and $b_s(y)$ are polynomials of degree $s-2$ for $s\ge2$. More
precisely, the two top degree terms of these polynomials are as follows
\begin{gather*}
 a_s(x)=(s-1)\tth
x^{s-2}+\frac{(s-1)(s-2)}2\,\tth(1-\tth)x^{s-3}\,+\dots,
\qquad s\ge 2\\
b_s(y)=-(s-1)\tth y^{s-2}+\frac{(s-1)(s-2)}2\,\tth(1-\tth)y^{s-3}\,+\dots,
\qquad s\ge 2.
\end{gather*}s
\end{lemma}

\begin{proof} Setting $v=u^{-1}$ we get
\begin{gather*} \frac{(u-x)(u-x+\tth-1)}{(u-x-1)(u-x+\tth)}=1+\frac{\tth
v}{1+\tth}\left(\frac1{1-(x+1)v}-\frac1{1-(x-\tth)v}\right)\\
=1+\frac{\tth}{1+\tth}\sum_{s\ge1}v^s\left((x+1)^{s-1}-(x-\tth)^{s-1}\right)\\
=1+\frac{\tth}{1+\tth}\sum_{s\ge2}v^s\left((s-1)(1+\tth)x^{s-2}
+\frac{(s-1)(s-2)}2(1-\tth^2)x^{s-3}+\dots\right)\\
=1+\sum_{s\ge2}v^s\left((s-1)\tth x^{s-2}
+\frac{(s-1)(s-2)}2\tth(1-\tth)x^{s-3}+\dots\right),
\end{gather*}
which proves the claim concerning the first expansion. For the second expansion
the computation is analogous:
\begin{gather*}
\frac{(u-y+1)(u-y-\tth)}{(u-y)(u-y-\tth+1)}=1-\frac{\tth
v}{1-\tth}\left(\frac1{1-yv}-\frac1{1-(y+\tth-1)v}\right)\\
=1-\frac{\tth v}{1-\tth}\sum_{s\ge1}v^{s-1}\left(y^{s-1}-(y+\tth-1)^{s-1}\right)\\
=1-\frac{\tth}{1-\tth}\sum_{s\ge2}v^s\left((s-1)(1-\tth)y^{s-2}
-\frac{(s-1)(s-2)}2(1-\tth)^2y^{s-3}+\dots\right)\\
=1+\sum_{s\ge2}v^s\left(-(s-1)\tth y^{s-2}
+\frac{(s-1)(s-2)}2\tth(1-\tth)y^{s-3}+\dots\right).
\end{gather*}

\end{proof}

For a partition $\si=(\si_1,\si_2,\dots)$ we set
\begin{equation}\label{tag7.9}
a_\si(x)=\prod_i a_{\si_i}(x), \qquad b_\si(y)=\prod_i b_{\si_i}(y).
\end{equation}
Note that these polynomials vanish if $\si$ has a part equal to 1, because
$a_1(x)$ and $b_1(x)$ are identically equal to 0.

Observe that \eqref{tag7.8} and \eqref{tag7.9} imply
\begin{equation}\label{tag7.10}
\begin{aligned}
\prod_l\frac{(u_l-x)(u_l-x+\tth-1)}{(u_l-x-1)(u_l-x+\tth)} =\sum_{\si}
a_\si(x)m_\si(u_1^{-1},u_2^{-1},\dots) \\
\prod_l\frac{(u_l-y+1)(u_l-y-\tth)}{(u_l-y)(u_l-y-\tth+1)}= \sum_{\si}
b_\si(y)m_\si(u_1^{-1},u_2^{-1},\dots).
\end{aligned}
\end{equation}

Next, introduce linear maps
\begin{gather*}
f\to\langle f\rangle^\up, \qquad \R[x]\to\A_\tth,\\
g\to\langle g\rangle^\down, \qquad \R[y]\to\A_\tth,
\end{gather*}
by setting
\begin{equation}\label{tag7.11}
\langle x^m\rangle^\up=\h_m, \qquad \langle y^m\rangle^\down=\he_{m+2}, \qquad
m=0,1,2,\dots, \quad \h_0:=1.
\end{equation}
This definition is inspired by Lemma \ref{6.11}.

Finally, let $c^\rho_{\si\tau}$ be the structure constants of the algebra $\La$
in the basis of monomial symmetric functions:
$$
m_\si m_\tau=\sum_\rho c^\rho_{\si\tau}m_\rho.
$$
Note that $c^\rho_{\si\tau}$ vanishes unless $|\rho|=|\si|+|\tau|$.

Now we are in a position to compute $U\h_\rho$ and $D\h_\rho$:

\begin{lemma}\label{7.4} With the notation introduced above we have
\begin{gather}
 U\h_\rho=\sum_{\si,\tau:\,
|\si|+|\tau|=|\rho|}c^\rho_{\si\tau}\left\langle(z+x)(z'+x)
a_\si(x)\right\rangle^\up\h_\tau, \label{tag7.12}\\
D\h_\rho=\sum_{\si,\tau:\, |\si|+|\tau|=|\rho|}c^\rho_{\si\tau}\left\langle
b_\si(y)\right\rangle^\down\h_\tau. \label{tag7.13}
\end{gather}
\end{lemma}

\begin{proof} Write
\begin{gather*} F^\up(u_1,u_2,\dots)=\sum_\si F^\up_\si
m_\si(u_1^{-1},u_2^{-1},\dots), \qquad
F^\up_\si\in\A_\tth,\\
F^\down(u_1,u_2,\dots)=\sum_\si F^\down_\si m_\si(u_1^{-1},u_2^{-1},\dots),
\qquad F^\down_\si\in\A_\tth.
\end{gather*}

{}From \eqref{tag7.2} and \eqref{tag7.7} we get
$$
\sum_\rho m_\rho(u_1^{-1},u_2^{-1},\dots)U\h_\rho=\left(\sum_\si F^\up_\si
m_\si(u_1^{-1},u_2^{-1},\dots)\right)\left(\sum_\tau
m_\tau(u_1^{-1},u_2^{-1},\dots)\h_\tau\right),
$$
which implies
$$
U\h_\rho=\sum_{\si,\tau:\,
|\si|+|\tau|=|\rho|}c^\rho_{\si\tau}F^\up_\si\h_\tau.
$$
Likewise,
$$
D\h_\rho=\sum_{\si,\tau:\,
|\si|+|\tau|=|\rho|}c^\rho_{\si\tau}F^\down_\si\h_\tau.
$$

It remains to prove that
$$
F^\up_\si=\left\langle(z+x)(z'+x) a_\si(x)\right\rangle^\up, \qquad
F^\down_\si=\left\langle b_\si(y)\right\rangle^\down,
$$
but this directly follows from \eqref{tag7.5}, \eqref{tag7.6}, and
\eqref{tag7.10}.

Finally, we note that $F^\up_\si$ and $F^\down_\si$ vanish if $\si$ has a part
equal to 1, because in this case $a_\si(x)\equiv0$ and $b_\si(y)\equiv0$. This
agrees with the remark made just below \eqref{tag7.2}.
\end{proof}

\subsection{Top degree terms of $D$: proof of claim (i) of Theorem
7.1}\label{7-4}

The existence of the operator $D:\A_\tth\to\A_\tth$ satisfying \eqref{tag7.1}
has been established above (see \eqref{tag7.7}), and its uniqueness is obvious.

By virtue of \eqref{tag7.13} we can write
\begin{equation}\label{tag7.14}
D=\sum_\si D_\si, \qquad D_\si\h_\rho=\sum_{\tau:\, |\tau|=|\rho|-|\si|}\langle
b_\si(y)\rangle^\down c^\rho_{\si\tau}\h_\tau.
\end{equation}

\begin{lemma}\label{7.5} Let $\si\ne\varnothing$. Then
\begin{equation}\label{tag7.15}
\deg D_\si\le\max_{\rho,\tau}(\ell(\rho)-\ell(\tau)-2\ell(\si)+1),
\end{equation}
where the maximum is taken over all pairs $(\rho,\tau)$ such that
$c^\rho_{\si\tau}\ne0$.

Furthermore, a more rough but simpler estimate is
\begin{equation}\label{tag7.16}
\deg D_\si\le -\ell(\si)+1.
\end{equation}
\end{lemma}

\begin{proof} We have
\begin{gather*}
 \deg D_\si\le \max_{\rho,\tau} \left(\deg\langle
b_\si(y)\rangle^\down+\deg\h_\tau-\deg\h_\rho \right)\\
=\max_{\rho,\tau} \left(\deg\langle
b_\si(y)\rangle^\down+|\tau|-\ell(\tau)-|\rho|+\ell(\rho) \right)\\
= \max_{\rho,\tau} \left(\deg\langle
b_\si(y)\rangle^\down-|\si|-\ell(\tau)+\ell(\rho) \right).
\end{gather*}
Here the first line holds by the very definition of $D_\si$, the second line
holds because
$$
\deg\h_\tau=|\tau|-\ell(\tau), \qquad \deg\h_\rho=|\rho|-\ell(\rho)
$$
for any $\tau$ and $\rho$, and the third line holds because
$c^\rho_{\si\tau}\ne0$ implies $|\rho|=|\si|+|\tau|$.

Let us write down $\langle b_\si(y)\rangle^\down$ in more detail. Set
$\si=(\si_1,\dots,\si_{\ell(\si)})$. Here $\ell(\si)\ge1$ because
$\si\ne\varnothing$ by the assumption. We may assume that $\si$ does not have
parts equal to 1; otherwise $D_\si=0$ because $b_1(y)\equiv0$. Thus,
$\si_i\ge2$ for all $i$ and we have
\begin{gather*}
\langle b_\si(y)\rangle^\down=\left\langle \prod_{i=1}^{\ell(\si)}
b_{\si_i}(y)\right\rangle^\down\\
 =\left\langle\prod_{i=1}^{\ell(\si)}(-(\si_i-1)\tth
y^{\si_i-2}+\frac{(\si_i-1)(\si_i-2)}2\,\tth(1-\tth)y^{\si_i-3}\,+\dots)
\right\rangle^\down,
\end{gather*}
where we have used Lemma \ref{7.3}.

The expression inside the brackets has degree $|\si|-2\ell(\si)$ in $y$.
Consequently, the top degree term of $\langle b_\si(y)\rangle^\down$ is equal,
within a nonzero scalar factor, to $\he_{|\si|-2\ell(\si)+2}$, and the degree
of this element is $|\si|-2\ell(\si)+1$.

Therefore,
$$
\deg D_\si\le\max_{\rho,\tau} (|\si|-2\ell(\si)-|\si|-\ell(\tau)+\ell(\rho)+1)=
\max_{\rho,\tau}(\ell(\rho)-\ell(\tau)-2\ell(\si)+1),
$$
which is \eqref{tag7.15}.

To deduce \eqref{tag7.16} we observe that $c^\rho_{\si\tau}\ne0$ implies
$\ell(\rho)\le \ell(\si)+\ell(\tau)$.
\end{proof}

\begin{corollary}\label{7.6} If $\ell(\si)\ge3$ then $\deg D_\si\le-2$.
\end{corollary}

\begin{proof} Indeed, this immediately follows from \eqref{tag7.16}.
\end{proof}

By Corollary \ref{7.6}, to prove claim {\rm(i)} of Theorem \ref{7.1} it
suffices to examine the contribution of the operators $D_\si$ with
$\ell(\si)=0$ (that is, $\si=\varnothing$), $\ell(\si)=1$, and $\ell(\si)=2$.
We do this in the three lemmas below.

\begin{lemma}[\rm Contribution from $\si=\varnothing$]\label{7.7}
$D_\varnothing=\h_2$.
\end{lemma}

\begin{proof} Indeed, if $\si=\varnothing$ then $\tau$ has to be equal to
$\rho$, and then $c^\rho_{\si\tau}=1$. On the other hand, $\langle
b_\si(y)\rangle^\down$ reduces to $\langle 1\rangle^\down=\he_2$.

Next, $\he_2=-\e_2$ (see Lemma \ref{6.12}) and the identity $h_2+e_2=h_1^2$ in
$\La$ implies the identity $-\e_2=\h_2+\h_1^2$ in $\A_\tth$. Since $\h_1=0$, we
conclude from \eqref{tag7.14} that $D_\varnothing$ is the operator of
multiplication by $\h_2$.
\end{proof}

\begin{lemma}[\rm Contribution from $\si$'s with $\ell(\si)=2$]\label{7.8}
$$
\sum_{\si:\, \ell(\si)=2}D_\si =\frac12\tth^2\sum_{r,s\ge2}(r-1)(s-1)
\h_{r+s-2}\frac{\pd^2}{\pd \h_r \pd\h_s} +\text{\rm terms of degree $\le-2$}.
$$
\end{lemma}

\begin{proof} Let $\ell(\si)=2$, so that $\si=(\si_1\ge\si_2>0)$. Recall that we
may assume $\si_2\ge2$ (otherwise $D_\si=0$). Below $\rho$ and $\tau$ are the
same as in Lemma \ref{7.5}. In particular, $\ell(\rho)\le\ell(\tau)+2$. If
$\ell(\rho)<\ell(\tau)+2$, then the argument of Lemma \ref{7.5} says that the
corresponding contribution to $D_\si$ has degree $\le-2$. Thus, we may take
into account only those $(\rho,\si)$ for which $\ell(\rho)=\ell(\tau)+2$. This
means $\rho=\si\cup\tau$, that is, the nonzero parts of $\rho$ are the disjoint
union of those in $\si$ and $\tau$. In other words, for some $i<j\le\ell(\rho)$
$$
\si_1=\rho_i, \quad \si_2=\rho_j, \quad
\tau=\{\rho_1,\dots,\rho_{\ell(\rho}\}\setminus\{\rho_i,\rho_j\}.
$$
In this case $c^\rho_{\si\tau}=1$

Furthermore, the argument in Lemma \ref{7.5} also shows that in $\langle
b_\si(y)\rangle^\down$, only the top degree term is relevant. This top degree
term is
$$
\langle(-(\si_1-1)\tth y^{\si_1-2}) (-(\si_2-1)\tth
y^{\si_2-2})\rangle^\down=\tth^2(\si_1-1)(\si_2-1)\he_{\si_1+\si_2-2}.
$$

It follows (see \eqref{tag7.12}) that
\begin{multline*}
\left(\sum_{\si:\, \ell(\si)=2}D_\si\right)\h_\rho =\tth^2\sum_{1\le
i<j\le\ell(\rho)}
(\rho_i-1)(\rho_j-1)\he_{\rho_i+\rho_j-2}\h_{\rho\setminus\{\rho_i,\rho_j\}}\\
+\text{\rm negligible terms}.
\end{multline*}

Therefore,
\begin{multline*}
\sum_{\si:\, \ell(\si)=2}D_\si =\tth^2\sum_{r_1>r_2\ge2}(r_1-1)(r_2-1)
\he_{r_1+r_2-2}\frac{\pd^2}{\pd
\h_{r_1}\pd\h_{r_2}}\\+\frac12\tth^2\sum_{r\ge2}(r-1)^2\he_{2r-2}\frac{\pd^2}{\pd
\h_r^2} +\text{\rm terms of degree $\le-2$}.
\end{multline*}

Observe that
$$
\he_r=\h_r+\text{\rm lower degree terms}, \quad r\ge2.
$$
Indeed, recall that $\he_r=(-1)^{r-1}\e_r$ (Lemma \ref{6.12}). In the algebra
$\La$, one has
\begin{multline}\label{tag7.17}
(-1)^{r-1}e_r=h_r-(h_1h_{r-1}+h_2h_{r-2}+\dots
+h_{r-1}h_1)\\+\text{\rm linear combination of triple, etc., products of $h_1,
h_2,\dots$}
\end{multline}
Projecting to $\A_\tth$ we get
\begin{multline}\label{tag7.18}
\he_r=\h_r-(\h_2\h_{r-2}+\dots +\h_{r-2}\h_2)\\+\text{\rm linear
combination of triple, etc., products of $\h_2, \h_3, \dots$},
\end{multline}
because $\h_1=0$. In \eqref{tag7.17}, all terms are homogeneous elements of
$\La$ of one and the same degree $r$. However, in \eqref{tag7.18} the only
terms of highest degree (with respect to the filtration of $\A_\tth$) are
$\he_r$ and $\h_r$. Consequently, replacing $\he_r$ by $\h_r$ affects only
negligible terms.

Thus, we get
\begin{multline*}
\sum_{\si:\, \ell(\si)=2}D_\si =\tth^2\sum_{r_1>r_2\ge2}(r_1-1)(r_2-1)
\h_{r_1+r_2-2}\frac{\pd^2}{\pd
\h_{r_1}\h_{r_2}}\\+\frac12\tth^2\sum_{r\ge2}(r-1)^2\h_{2r-2}\frac{\pd^2}{\pd
\h_r^2} +\text{\rm terms of degree $\le-2$},
\end{multline*}
which is equivalent to the desired expression.
\end{proof}

\begin{lemma}[\rm Contribution from $\si$'s with $\ell(\si)=1$]\label{7.9}
\begin{gather*}
\sum_{\si:\, \ell(\si)=1}D_\si =-\tth\sum_{r\ge2}(r-1)\h_r\frac{\pd}{\pd\h_r}\\
+\frac12\tth(1-\tth)\sum_{r\ge3}(r-1)(r-2)\h_{r-1}\frac{\pd}{\pd\h_r}\\
+\frac12\tth\sum_{r,s\ge2}(r+s)\h_r\h_s\frac{\pd}{\pd\h_{r+s}}\\
+\text{\rm terms of degree $\le-2$}.
\end{gather*}
\end{lemma}

\begin{proof} Let $\ell(\si)=1$, so that $\si=(s)$ with $s\ge2$. Below $\rho$
and $\tau$ are the same as in Lemma \ref{7.5}. Two cases are possible:
$\ell(\tau)=\ell(\rho)-1$ and $\ell(\tau)=\ell(\rho)$. Let us examine them
separately.

Assume $\ell(\tau)=\ell(\rho)-1$. This means that, for some
$i=1,\dots,\ell(\rho)$, we have $\si=(\rho_i)$  and
$\tau=\rho\setminus\{\rho_i\}$. Note that then $c^\rho_{\si\tau}=1$. We argue
as in the proof of Lemma \ref{7.8}, the only difference is that we have to take
into account not only the top degree term in $\langle b_\si(y)\rangle^\down$
but also the next term. Thus, applying Lemma \ref{7.3}, we write
$$
\langle b_\si(y)\rangle^\down=\langle b_s(y)\rangle^\down=-(s-1)\tth
\he_s+\frac{(s-1)(s-2)}2\,\tth(1-\tth)\he_{s-1}\,+\dots\,.
$$
According to \eqref{tag7.14}, this gives rise to the terms
$$
-\tth\sum_{r\ge2}(r-1)\he_r\frac{\pd}{\pd\h_r}
+\frac12\tth(1-\tth)\sum_{r\ge3}(r-1)(r-2)\he_{r-1}\frac{\pd}{\pd\h_r}\,.
$$

Now, assume $\ell(\tau)=\ell(\rho)$. This means that $\tau$ is obtained from
$\rho$ by subtracting $s$ from one of the parts $\rho_i$ of $\rho$; moreover,
this part $\rho_i$ should be $\ge s+2$. Note that $c^\rho_{\si\tau}$ is just
equal to the multiplicity of that part in $\rho$. Note also that only the top
degree term in $\langle b_s(y)\rangle^\down$ has a relevant contribution. This
gives rise to the terms
$$
-\tth\sum_{r\ge4,\,2\le s\le r-2}(s-1)\he_s\h_{r-s}\frac{\pd}{\pd\h_r}\,.
$$

Next, the above two expressions involve $\he_r$, $\he_{r-1}$, and
$\he_s$, which we have to express in terms of $\h_i$'s. This should
be done as follows:
$$
\he_r=\h_r-(\h_2\h_{r-2}+\dots +\h_{r-2}\h_2)+\dots, \qquad
\he_{r-1}=\h_{r-1}+\dots, \qquad \he_s=\h_s+\dots,
$$
where the rest terms denoted by dots contribute only to terms of degree $\le-2$
in $D$. Collecting all the terms together and slightly changing the notation of
indices we get
\begin{gather*}
-\tth\sum_{r\ge2}(r-1)\h_r\frac{\pd}{\pd\h_r}+\tth\sum_{r\ge2,
s\ge2}(r+s-1)\h_r\h_s\frac{\pd}{\pd\h_{r+s}}\\
+\frac12\tth(1-\tth)\sum_{r\ge3}(r-1)(r-2)\h_{r-1}\frac{\pd}{\pd\h_r}
-\tth\sum_{r\ge2, s\ge2}(s-1)\h_r\h_s\frac{\pd}{\pd\h_{r+s}}\,.
\end{gather*}
Now, putting together the second and fourth sums, there is a simplification,
which finally leads to the desired expression.
\end{proof}

The expressions obtained in Lemmas \ref{7.7}, \ref{7.8}, and \ref{7.9} give
together the result stated in claim {\rm(i)} of Theorem \ref{7.1}.

\subsection{Top degree terms of $U$: proof of claim (ii) of Theorem
7.1}\label{7-5}

The strategy of the proof is the same as in the preceding subsection. However,
we have to slightly modify our arguments because of the following
circumstances:

$\bullet$ Formula \eqref{tag7.12}, as compared to formula \eqref{tag7.13},
contains the additional factors $(z+x)(z'+x)$.

$\bullet$ As seen from \eqref{tag7.11}, there is a subtle difference in the
behavior of the degree of $\langle x^m\rangle^\up=\h_m$ and the degree of
$\langle y^m\rangle^\down=\he_{m+2}$. For the latter quantity we have a
``regular'' expression $\deg \langle y^m\rangle^\down=m+1$, valid for all
$m\ge0$, while a similar expression for the former quantity, $\deg \langle
x^m\rangle^\up=m-1$, holds for $m\ge1$ but fails for $m=0$.

In accordance with \eqref{tag7.12}, it is convenient to decompose $U$ as
follows
$$
U=\sum_\si(U^0_\si+U^1_\si+U^2_\si),
$$
where
\begin{gather} U^0_\si\h_\rho=zz'\sum_\tau\langle a_\si(x)\rangle^\up
c^\rho_{\si\tau}\h_\tau \label{tag7.19}\\
U^1_\si\h_\rho=(z+z')\sum_\tau\langle a_\si(x)x\rangle^\up
c^\rho_{\si\tau} \h_\tau \label{tag7.20}\\
U^2_\si\h_\rho=\sum_\tau\langle a_\si(x)x^2\rangle^\up c^\rho_{\si\tau}
\h_\tau. \label{tag7.21}
\end{gather}

\begin{lemma}[\rm Compare to Lemma \ref{7.5}]\label{7.10}
Let $\si\ne\varnothing$. Then the following estimate for $\deg U^1_\si$ and
$\deg U^2_\si$ holds
$$ \deg
U^p_\si\le\max_{\rho,\tau}(\ell(\rho)-\ell(\tau)-2\ell(\si)-1+p), \qquad p=1,2,
$$
where the maximum is taken over all pairs $(\rho,\tau)$ such that
$c^\rho_{\si\tau}\ne0$.

Furthermore, a more rough but simpler estimate is
$$
\deg U^p_\si\le -\ell(\si)-1+p, \qquad p=1,2.
$$
\end{lemma}

Notice that the case of $U^0_\si$ requires a special investigation.

\begin{proof}
The argument is completely similar to that in Lemma \ref{7.5}. We have
\begin{gather*}
 \deg U^p_\si\le \max_{\rho,\tau} \left(\deg\langle
a_\si(x)x^p\rangle^\up+\deg\h_\tau-\deg\h_\rho \right)\\
=\max_{\rho,\tau} \left(\deg\langle
a_\si(x)x^p\rangle^\up+|\tau|-\ell(\tau)-|\rho|+\ell(\rho) \right)\\
=\max_{\rho,\tau} \left(\deg\langle
a_\si(x)x^p\rangle^\up-|\si|-\ell(\tau)+\ell(\rho) \right)
\end{gather*}

Since $p>0$ by the assumption, the polynomial $a_\si(x)x^p$ has degree $>0$
even if the polynomial $a_\si(x)$ is a constant. Therefore,
$$
\deg\langle a_\si(x)x^p\rangle^\up=|\si|-2\ell(\si)+p-1,
$$
which gives the first estimate. Then the second estimate follows
from the inequality $\ell(\rho)\le\ell(\si)+\ell(\tau)$.

\end{proof}

\begin{corollary}[\rm Compare to Corollary \ref{7.6}]\label{7.11}
We have:

{\rm(i)} $\deg U^1_\si\le-2$ if $\ell(\si)\ge2$;

{\rm(ii)} $\deg U^2_\si\le-2$ if $\ell(\si)\ge3$;
\end{corollary}

\begin{proof} Indeed, this follows at once from the second estimate in Lemma
\ref{7.10}.
\end{proof}

We will examine the cases $p=0$, $p=1$, and $p=2$ separately.

\begin{lemma}\label{7.12} {\rm(Contribution from $U^0_\si$'s)}
$$
\sum_\si U^0_\si=zz'+\tth zz'\frac{\pd}{\pd\h_2}+\text{\rm terms of degree
$\le-2$}.
$$
\end{lemma}

\begin{proof} The contribution of $U^0_\varnothing$ is the constant term $zz'$:
this is shown by the same argument as in Lemma \ref{7.7}.

Assume $\si\ne\varnothing$. Then $\si=(\si_1,\dots,\si_{\ell(\si)})$ with
$\ell(\si)\ge1$. Recall that all $\si_i$ are $\ge2$. Arguing as in Lemma
\ref{7.10} we get
\begin{gather*} \deg U^0_\si\le\max_{\rho,\tau} \left(\deg\langle
a_\si(x)x^p\rangle^\up+|\tau|-\ell(\tau)-|\rho|+\ell(\rho) \right)\\
\le \max_{\rho,\tau}(\deg\langle a_\si(x)\rangle^\up-|\si|+\ell(\si)),
\qquad\text{\rm because $-\ell(\tau)+\ell(\rho)\le\ell(\si)$}.
\end{gather*}
Here $\rho$ and $\tau$ are the same as in Lemma \ref{7.10}.

In the ``regular case'', when the polynomial $a_\si(x)$ has degree $>0$, we can
apply the formula $\deg\langle a_\si(x)\rangle^\up=|\si|-2\ell(\si)-1$, which
implies
$$
\deg U^0_\si\le -\ell(\si)-1\le-2.
$$
The ``irregular case'' occurs when $\si_1=\dots=\si_{\ell(\si)}=2$ (see Lemma
\ref{7.3}). Then $\deg\langle a_\si(x)\rangle^\up=0$ and we get a weaker
inequality
$$
\deg U^0_\si\le -\ell(\si).
$$
If $\ell(\si)\ge2$, this is enough to conclude $\deg U^0_\si\le -2$.

Finally, examine the case $\si=(2)$. There are two possibilities:
$\ell(\tau)=\ell(\rho)-1$ and $\ell(\tau)=\ell(\rho)$. In the latter case the
estimate can be refined because then $-\ell(\tau)+\ell(\rho)=0$ is strictly
smaller than $\ell(\si)=1$, which again implies $\deg U^0_\si\le -2$.

Thus, the only substantial contribution arises when $\si=(2)$ and
$\tau=\rho\setminus\{(2)\}$. Taking into account \eqref{tag7.19} and Lemma
\ref{7.3}, this gives rise to the term $\tth zz' \pd/\pd\h_2$.
\end{proof}

\begin{lemma}[\rm Contribution from $U^1_\si$'s]\label{7.13}
$$
\sum_\si U^1_\si=\tth(z+z')\sum_{r\ge3}(r-1)\h_{r-1}\frac{\pd}{\pd\h_r}
+\text{\rm terms of degree $\le-2$}.
$$
\end{lemma}

\begin{proof} Observe that $U^1_\varnothing=0$ because $\langle
a_\varnothing(x)x\rangle^\up=\h_1=0$. By Corollary \ref{7.11}, it suffices to
examine the case $\ell(\si)=1$, that is, $\si=(s)$ with $s\ge2$. We have
$$
\langle a_s(x)x\rangle^\up=(s-1)\tth \h_{s-1}+\dots,
$$
where the rest terms are negligible. Furthermore, if $\ell(\tau)=\ell(\rho)$
then the estimate of Corollary \ref{7.11} can be refined, which implies that
the contribution is negligible. Thus, we may assume $\ell(\tau)=\ell(\rho)-1$,
that is, $\tau=\rho\setminus\{(s)\}$. In accordance with \eqref{tag7.20}, this
produces the desired expression. Notice also that the restriction $r\ge3$
arises because $\h_{r-1}=0$ for $r=2$.
\end{proof}

It remains to compute $\sum_\si U^2_\si$. Here our arguments are
strictly parallel to those of the preceding subsection, because, due
to the extra factor $x^2$, the element $\langle
a_\si(x)x^2\rangle^\up$ has the same degree as the element $\langle
b_\si(x)\rangle^\down$, which we examined in the preceding
subsection.

In the next three lemmas we rely on \eqref{tag7.21}.

\begin{lemma}\label{7.14} {\rm(Contribution from $U^2_\varnothing$)}
$U^2_\varnothing=\h_2$.
\end{lemma}

\begin{proof}
The same argument as in Lemma \ref{7.7}. The situation is even simpler
because we do not need to convert $\he_2$ to $\h_2$. \
\end{proof}

\begin{lemma}[\rm Contribution from $U^2_\si$'s with
$\ell(\si)=2$]\label{7.15}
$$
\sum_{\si:\, \ell(\si)=2}U^2_\si =\frac12\tth^2\sum_{r,s\ge2}(r-1)(s-1)
\h_{r+s-2}\frac{\pd^2}{\pd \h_r\h_s} +\text{\rm terms of degree $\le-2$}
$$
\end{lemma}

\begin{proof} The argument is exactly the same as in the proof of Lemma \ref{7.8}.
Instead of Lemma \ref{7.5} we refer to its analog, Lemma \ref{7.10}.  Again, we
do not need to convert $\he_r$ to $\h_r$, which slightly shortens the proof.
\end{proof}

\begin{lemma}[\rm Contribution from $U^2_\si$'s with
$\ell(\si)=1$]\label{7.16}
\begin{gather*}
\sum_{\si:\, \ell(\si)=1}U^2_\si=
\tth\sum_{r\ge2}(r-1)\h_r\frac{\pd}{\pd\h_r}\\
+\frac12\tth(1-\tth)\sum_{r\ge3}(r-1)(r-2)\h_{r-1}\frac{\pd}{\pd\h_r}\\
+\frac12\tth\sum_{r,s\ge2}(r+s-2)\h_r\h_s\frac{\pd}{\pd\h_{r+s}} +\text{\rm
terms of degree $\le-2$}.
\end{gather*}
\end{lemma}

\begin{proof} We argue as in the proof of Lemma \ref{7.9}.
\end{proof}

Lemmas \ref{7.14}, \ref{7.15}, and \ref{7.16} together give
\begin{gather*}
U^2 =\h_2+\frac12\tth^2\sum_{r,s\ge2}(r-1)(s-1)
\h_{r+s-2}\frac{\pd^2}{\pd \h_r\pd\h_s}\\
+\tth\sum_{r\ge2}(r-1)\h_r\frac{\pd}{\pd\h_r}\\
+\frac12\tth(1-\tth)\sum_{r\ge3}(r-1)(r-2)\h_{r-1}\frac{\pd}{\pd\h_r}\\
+\frac12\tth\sum_{r,s\ge2}(r+s-2)\h_r\h_s\frac{\pd}{\pd\h_{r+s}}\\
+\text{terms of degree $\le-2$}.
\end{gather*}
Adding this with the expressions obtained in Lemmas \ref{7.12} and \ref{7.13}
we finally get the result indicated in Claim {\rm(ii)} of Theorem \ref{7.1}.

This completes the proof of Theorem \ref{7.1}.

\section{Computation of $T_n-1$}\label{8}

As in Section \ref{7}, here we are dealing with a fixed triple $(\tth,z,z')$ of
parameters. As usual, we assume $\tth>0$. As for the couple $(z,z')$, we now
assume that it belongs to the principal or complementary series (Definition
\ref{5.4}), so that $p^\up_{\tth,z,z'}$ is a true system of transition
probabilities.

\begin{definition}[\rm The operator $T_n$]\label{8.1}
Recall that in Definition \ref{2.3} we have introduced the up-down Markov
chains associated with arbitrary systems $p^\up$ and $p^\down$ of up and down
transition probabilities. Now we take the concrete systems
$p^\up=p^\up_{\tth,z,z'}$, $p^\down=p^\down_\tth$ and denote by $T_n$ the
transition operator of the corresponding up-down chain of level $n$,
$n=1,2,\dots$\,. We regard $T_n$ as an operator in the space $\Fun(\Y_n)$ of
functions on $\Y_n$:
$$
T_nF(\la)=\sum_{\nu,\kappa}p^\up_{\tth,z,z'}(\la,\nu)p^\down_\tth(\nu,\kappa)F(\kappa),
\qquad F\in\Fun(\Y_n),
$$
summed over all $\nu\in\Y_{n+1}$ and $\kappa\in\Y_n$ such that
$\la\nearrow\nu\searrow\kappa$.
\end{definition}

In the notation introduced in the beginning of Section \ref{7},
$$
T_n=U_{n,n+1}D_{n+1,n}\,.
$$
Set
\begin{equation}\label{tag8.1}
\epsi_n=\frac1{(\tth^{-1}zz'+n)(n+1)}\,, \qquad n=1,2,\dots\,.
\end{equation}
Equivalently,
$$
\epsi_n^{-1}=\frac{(zz'+\tth n)(\tth(n+1))}{\tth^2}\,.
$$
Clearly, $\epsi_n\sim n^{-2}$ (in Theorem \ref{9.6} below we simply take
$\epsi_n=n^{-2}$). Recall that given a function $F\in\A_\tth$, we denote by
$F_n\in\Fun(\Y_n)$ the restriction of $F$ to the finite subset $\Y_n$.

\begin{theorem}\label{8.2} There exists a unique operator $\wt B:\A_\tth\to\A_\tth$
such that for any $F\in\A_\tth$ and each $n=1,2,\dots$
$$
\epsi^{-1}_n(T_n-1)F_n=(\wt B F)_n\,.
$$
The operator $\wt B$ has degree $0$ and its top degree component looks as
follows
\begin{multline}\label{tag8.2}
\wt B= \sum_{r,s\ge3}(r-1)(s-1)
(\h_2\h_{r+s-2}-\h_r\h_s)\frac{\pd^2}{\pd \h_r\pd\h_s}\\
+\sum_{r\ge3}\big[(\tth^{-1}-1)(r-1)(r-2)\h_2\h_{r-1}+\tth^{-1}(z+z')(r-1)\h_2\h_{r-1}\\
-(r-1)(r-2)\h_r-\tth^{-1}zz'(r-1)\h_r\big]\frac{\pd}{\pd\h_r}\\
+\tth^{-1}\sum_{r,s\ge2}(r+s-1)\h_2\h_r\h_s\frac{\pd}{\pd\h_{r+s}} +\text{\rm
terms of degree $<0$}.
\end{multline}
\end{theorem}

\begin{proof} The uniqueness claim is obvious. The existence of $\wt B$ follows
from Theorem \ref{7.1}. Indeed, by virtue of \eqref{tag6.12} and Proposition
\ref{5.1}, $\h_2(\la)=\tth n$. Using this and expressing $U_{n,n+1}$ and
$D_{n+1,n}$ through $U$ and $D$, as indicated in Theorem \ref{7.1}, we get
\begin{gather*}
(T_n-1)F_n=U_{n,n+1}D_{n+1,n}F_n-F_n =\frac{(UDF)_n-(zz'+\tth
n)(\tth(n+1))F_n}{(zz'+\tth n)(\tth(n+1))}
\\ =\frac{((UD-(\h_2+zz')(\h_2+\tth))F)_n}{(zz'+\tth n)(\tth(n+1))}\,.
\end{gather*}
Therefore,
\begin{equation}\label{tag8.3}
\wt B=\tth^{-2}(UD-(\h_2+zz')(\h_2+\tth)).
\end{equation}

To compute $UD-(\h_2+zz')(\h_2+\tth)$, within negligible terms, we write
$$
U=\h_2+U_0+U_{-1}+\dots, \qquad D=\h_2+D_0+D_{-1}+\dots,
$$
where $U_0$ and $D_0$ are the terms of degree 0 and $U_{-1}$ and
$D_{-1}$ are the terms of degree $-1$.

Note that $\h_2$ (more precisely, the operator of multiplication by $\h_2$) has
degree 1. Let us check that the operator $UD-(\h_2+zz')(\h_2+\tth)$, which
could have degree 2, is actually of degree 0, due to cancelation of the terms
of degree 2 and 1.

Indeed, the degree 2 terms in $UD-(\h_2+zz')(\h_2+\tth)$ are
$$
\h_2^2-\h_2^2=0.
$$

The degree 1 terms in $UD-(\h_2+zz')(\h_2+\tth)$ are
$$
\h_2D_0+U_0\h_2-(zz'+\tth)\h_2\,.
$$
Substitute here the explicit expressions for $D_0$ and $U_0$ taken from Theorem
\ref{7.1},
\begin{gather*} D_0=-\tth\sum_{r\ge2}(r-1)\h_r\frac{\pd}{\pd \h_r}\\
U_0=zz'+\tth\sum_{r\ge2}(r-1)\h_r\frac{\pd}{\pd \h_r}\,.
\end{gather*}
Since
$$
U_0\h_2=\h_2 U_0+[U_0,\h_2]=\h_2 U_0+\tth\h_2,
$$
all the terms of degree 1 are indeed cancelled out.

Now, let us examine the degree 0 terms in $UD-(\h_2+zz')(\h_2+\tth)$. These are
$$
\h_2D_{-1}+U_{-1}\h_2+U_0D_0-\tth zz'.
$$
To compute $\h_2D_{-1}+U_{-1}\h_2$ we substitute the explicit expressions for
$D_{-1}$ and $U_{-1}$ taken from Theorem \ref{7.1},
\begin{gather*}
D_{-1}=\frac12\tth^2\sum_{r,s\ge2}(r-1)(s-1)
\h_{r+s-2}\frac{\pd^2}{\pd \h_r\pd\h_s}\\
+\frac12\tth(1-\tth)\sum_{r\ge3}(r-1)(r-2)\h_{r-1}\frac{\pd}{\pd\h_r}\\
+\frac12\tth\sum_{r,s\ge2}(r+s)\h_r\h_s\frac{\pd}{\pd\h_{r+s}}
\end{gather*}
and
\begin{gather*}U_{-1}=\tth zz'\frac{\pd}{\pd\h_2}
+\tth(z+z')\sum_{r\ge3}(r-1)\h_{r-1}\frac{\pd}{\pd\h_r}\\
+\frac12\tth^2\sum_{r,s\ge2}(r-1)(s-1)
\h_{r+s-2}\frac{\pd^2}{\pd \h_r\pd\h_s}\\
+\frac12\tth(1-\tth)\sum_{r\ge3}(r-1)(r-2)\h_{r-1}\frac{\pd}{\pd\h_r}\\
+\frac12\tth\sum_{r,s\ge2}(r+s-2)\h_r\h_s\frac{\pd}{\pd\h_{r+s}}\,.
\end{gather*}

Note that
$$
[U_{-1},\h_2]=\tth zz'+\tth^2\sum_{r\ge2}(r-1)\h_r\frac{\pd}{\pd \h_r}\,.
$$
Therefore,
\begin{gather*} \h_2D_{-1}+U_{-1}\h_2=\h_2D_{-1}+\h_2 U_{-1}+[U_{-1},\h_2]\\
=\tth^2\sum_{r,s\ge2}(r-1)(s-1)
\h_2\h_{r+s-2}\frac{\pd^2}{\pd \h_r\pd\h_s}\\
+\tth zz'\h_2\frac{\pd}{\pd\h_2}\\
+\tth(z+z')\sum_{r\ge3}(r-1)\h_2\h_{r-1}\frac{\pd}{\pd\h_r}\\
+\tth(1-\tth)\sum_{r\ge3}(r-1)(r-2)\h_2\h_{r-1}\frac{\pd}{\pd\h_r}\\
+\tth\sum_{r,s\ge2}(r+s-1)\h_2\h_r\h_s\frac{\pd}{\pd\h_{r+s}}\\
+\tth^2\sum_{r\ge2}(r-1)\h_r\frac{\pd}{\pd\h_r}+\tth zz'.
\end{gather*}

Next, using the above expressions for $U_0$ and $D_0$ we get
\begin{gather*}
U_0D_0-\tth zz'=-\tth^2\left(\sum_{r\ge2}(r-1)\h_r\frac{\pd}{\pd\h_r}\right)^2
-\tth zz'
\sum_{r\ge2}(r-1)\h_r\frac{\pd}{\pd\h_r} -\tth zz'\\
=-\tth^2\sum_{r,s\ge2}(r-1)(s-1)\h_r\h_s\frac{\pd^2}{\pd\h_r\pd\h_s}
-\tth^2\sum_{r\ge2}(r-1)^2\h_r\frac{\pd}{\pd\h_r}\\
-\tth zz' \sum_{r\ge2}(r-1)\h_r\frac{\pd}{\pd\h_r} -\tth zz'.
\end{gather*}

Adding together the above two expressions we see that the terms involving
$\pd/\pd\h_2$ cancel out, which plays the crucial role in Corollary \ref{8.7}
below. Then we divide by $\tth^2$ in accordance with \eqref{tag8.3} and finally
get \eqref{tag8.2}.
\end{proof}

{}From the definition of the filtration in $\A_\tth$ (Definition \ref{6.2}) we
see that there is a natural isomorphism between the associated graded algebra
$$
\gr\A_\tth=\bigoplus_{m=0}^\infty(\A_\tth^{(m)}/\A_\tth^{(m-1)})
$$
and the algebra $\La=\R[p_1,p_2,\dots]$ of symmetric functions: the top degree
terms of the generators $p^*_m\in\A_\tth$ are identified with the homogeneous
generators $p_m\in\La$. This isomorphism $\gr\A_\tth\simeq\La$ should not be
confused with the covering homomorphism $\La\to\A_\tth$.

\begin{definition}[\rm The operator $B$]\label{8.3}
 Observe that the operator $\wt B:\A_\tth\to\A_\tth$ (Theorem
\ref{8.2}) has degree $0$. Therefore, it gives rise to an operator in the
associated graded algebra. Using the identification $\gr\A_\tth=\La$, we denote
the latter operator as $B:\La\to\La$. Note that $B$ is homogeneous of degree
$0$
\end{definition}

\begin{theorem}\label{8.4}
The operator $B:\La\to\La$ just defined has the following form
\begin{multline}\label{tag8.4}
B= \sum_{k,l\ge2}kl
(p_1 p_{k+l-1}-p_kp_l)\frac{\pd^2}{\pd p_k \pd p_l}\\
+\sum_{k\ge2}\big[(1-\tth)k(k-1)p_1 p_{k-1}+(z+z')k p_1 p_{k-1}
-k(k-1)p_k-k\tth^{-1}zz'p_k\big]\frac{\pd}{\pd p_k}\\
+\tth\sum_{k,l\ge1}(k+l+1)p_1 p_k p_l\frac{\pd}{\pd p_{k+l+1}}\,.
\end{multline}
\end{theorem}

Before proving the theorem let us state a lemma.

\begin{lemma}\label{8.5} Consider the algebra of polynomials in countably many
generators $\R[x_1,x_2,\dots]$ with the filtration determined by setting\/
$\deg x_k=k$, and let $y_1,y_2,\dots$ be another sequence of elements of the
same algebra such that
$$
x_k=y_k+\text{\rm terms of degree $<k$}, \qquad k=1,2,\dots,
$$
so that $\{y_1,y_2,\dots\}$ is also a system of generators.

Then we have
$$
\frac{\pd}{\pd x_k}=\frac{\pd}{\pd y_k}+\text{\rm terms of degree $<-k$},
\qquad k=1,2,\dots\,.
$$
\end{lemma}

\begin{proof}[Proof of Lemma \ref{8.5}] Indeed,
$$
\frac{\pd}{\pd x_k}=\sum_{l\ge1} \frac{\pd y_l}{\pd x_k}\,\frac{\pd}{\pd
y_l}\,.
$$
Since
$$
y_l=x_l+R_l, \qquad \text{\rm where $\deg R_l<l$},
$$
we have
$$
\frac{\pd}{\pd x_k}=\sum_{l\ge k}\frac{\pd y_l}{\pd x_k}\,\frac{\pd}{\pd
y_l}=\frac{\pd}{\pd y_k}+\sum_{l>k}\frac{\pd R_l}{\pd x_k}\,\frac{\pd}{\pd
y_l}\,.
$$
The degree of the $l$th summand in the last sum is strictly less than
$(l-k)-l=-k$, so that all these summands are negligible.
\end{proof}

\begin{proof}[Proof of Theorem \ref{8.4}] By virtue of \eqref{tag6.3} and \eqref{tag6.5},
and because of $\h_1=\p_1=0$,
$$
1+\sum_{k\ge1}\h_{k+1}u^{-k-1}
=\exp\left(\sum_{k\ge1}\frac{\p_{k+1}}{k+1}\,u^{-k-1}\right).
$$
It follows that
$$
\h_{k+1}=\frac{\p_{k+1}}{k+1}\,+\text{terms of degree $<k$}, \qquad
k=1,2,\dots\,.
$$
Next, by \eqref{tag6.8}
$$
\frac{\p_{k+1}}{k+1}=\tth p^*_k+\text{terms of degree $<k$}, \qquad
k=1,2,\dots\,.
$$
Therefore,
\begin{equation}\label{tag8.5}
\h_{k+1}=\tth p^*_k+\text{terms of degree $<k$}, \qquad k=1,2,\dots\,.
\end{equation}

It follows that we may apply Lemma \ref{8.5} to the generators $x_k=\h_{k+1}$
and $y_k=\tth p^*_k$, $k=1,2,\dots$\,. This gives us
\begin{equation}\label{tag8.6}
\frac{\pd}{\pd \h_{k+1}}=\frac1\tth\,\frac{\pd}{\pd p^*_k}+\text{terms of
degree $<-k$}, \qquad k=1,2,\dots\,.
\end{equation}

Substituting \eqref{tag8.5} and \eqref{tag8.6} into \eqref{tag8.2} we get a
similar expression for the operator $\wt B$ in terms of generators
$p^*_1,p^*_2,\dots$ and the corresponding partial derivatives.

Finally, it is readily seen that to get $B$ it suffices to replace $p^*_k$ with
$p_k$. This leads to \eqref{tag8.4}.
\end{proof}

\begin{definition}[\rm The quotient algebra $\La^\circ$]\label{8.6}
Consider the principal ideal $(p_1-1)\La$ in the algebra $\La$ generated by the
element $p_1-1$ and let $\La^\circ=\La/(p_1-1)\La$ denote the corresponding
quotient algebra. Because the ideal is not homogeneous, there is no natural
graduation in $\La^\circ$, but $\La^\circ$ inherits the filtration of $\La$.
Given $\varphi\in\La$, we denote by $\varphi^\circ$ the image of $\varphi$
under the canonical projection $\La\to\La^\circ$. Due to the natural
isomorphism of $\La^\circ$ with the polynomial algebra
$\R[p^\circ_2,p^\circ_3,\dots]$ we may introduce in $\La^\circ$ the
differential operators $\pd/\pd p^\circ_k$, $k\ge2$.
\end{definition}

\begin{corollary}\label{8.7}
The operator $B:\La\to\La$ introduced in Definition \ref{8.3} and computed in
Theorem \ref{8.4} preserves the principal ideal $(p_1-1)\La\subset\La$ and
hence gives rise to an operator $A:\La^\circ\to\La^\circ$. We have
\begin{multline} \label{tag8.7}
A= \sum_{k,l\ge2}kl
(p^\circ_{k+l-1}-p^\circ_kp^\circ_l)\frac{\pd^2}{\pd p^\circ_k \pd p^\circ_l}\\
+\sum_{k\ge2}\big[(1-\tth)k(k-1)p^\circ_{k-1}+(z+z')k p^\circ_{k-1}
-k(k-1)p^\circ_k-k\tth^{-1}zz'p^\circ_k\big]\frac{\pd}{\pd p^\circ_k}\\
+\tth\sum_{k,l\ge1}(k+l+1)p^\circ_k p^\circ_l\frac{\pd}{\pd p^\circ_{k+l+1}}
\end{multline}
with the understanding that $p^\circ_1=1$.
\end{corollary}

\begin{proof} This immediately follows from Theorem \ref{8.4} because the expression
\eqref{tag8.4} for the operator $B$ does not contain $\pd/\pd p_1$.
\end{proof}

\section{Construction of Markov processes}\label{9}

\subsection{An operator semigroup approximation theorem}\label{9-1}
We start with the statement of a well-known general result on approximation of
continuous contraction semigroups by discrete ones. Our basic reference is the
book \cite{EK2} by Ethier and Kurtz, where one can also find references to
original papers.

Assume we are given real Banach spaces $L, L_1,L_2,\dots$ together with bounded
linear operators $\pi_n:L\to L_n$, $n=1,2,\dots$, such that
$\sup_n\Vert\pi_n\Vert<\infty$.

\begin{definition}[\rm Convergence of vectors in varying Banach spaces]\label{9.1}
Let $f\in L$ and $f_n\in L_n$, $n=1,2,\dots$\,. Write $f_n\to f$ if
$$
\lim_{n\to\infty}\Vert f_n-\pi_n f\Vert=0.
$$
\end{definition}

In particular, if $f_n=\pi_n f$ then $f_n\to f$ for trivial reasons. Clearly,
if $f_n\to f$ and $g_n\to g$ then $f_n+g_n\to f+g$. Generally speaking, it may
happen that $f_n\to f$ and $f_n\to g$ with $f\ne g$. However, such an
unpleasant situation can be excluded by imposing an extra assumption like
\eqref{tag9.4} below. In the concrete situation we will dealing with,
\eqref{tag9.4} is satisfied but this is not required for Theorem \ref{9.3}
below.

\begin{definition}[\rm Approximation of operator semigroups]\label{9.2}
Let $L$, $\{L_n\}$, and $\{\pi_n\}$ be as above; $\{T(t)\}_{t\ge0}$ be a
strongly continuous contraction semigroup in $L$; $T_n$ be a contraction in
$L_n$, $n=1,2,\dots$; $\{\epsi_n\}$ be a sequence of numbers such that
$\epsi_n>0$ and $\epsi_n\to0$. Each contraction $T_n$ generates a discrete
semigroup, $\{T_n^m\}_{m=0,1,\dots}$, in $L_n$.

Let us say that these discrete semigroups {\it approximate\/}, as $n\to\infty$,
the continuous semigroup $\{T(t)\}$ if for any $f\in L$
$$
T_n^{[t\epsi^{-1}_n]}\pi_n f\to T(t)f
$$
uniformly on arbitrarily large bounded intervals $0\le t\le t_0$. That is,
$$
\Vert T_n^{[t\epsi^{-1}_n]}\pi_n f-\pi_n T(t)f\Vert \to0
$$
uniformly on $t\in[0,t_0]$.
\end{definition}

Let us emphasize that the definition substantially depends on the sequence
$\{\epsi_n\}$ which determines the time scaling.

\begin{theorem}\label{9.3}
Consider the same data as in Definition \ref{9.2}, and set
\begin{equation}\label{tag9.1}
A_n=\epsi^{-1}_n(T_n-1), \qquad n=1,2,\dots\,.
\end{equation}
Let $\F\subset L$ be a dense subspace and $A:\F\to L$ be a closable operator
such that its closure $\bar A$ coincides with the generator of the semigroup
$\{T(t)\}$. Next, let us assume that the operators $A_n$ converge to the
operator $A$ in the following ``extended'' sense\/{\rm:} For any $f\in\F$ there
exists a sequence $\{f_n\in L_n\}$ such that
\begin{equation}\label{tag9.2}
f_n\to f \quad \text{\rm and} \quad A_nf_n\to Af.
\end{equation}

Then the semigroups $\{T^m_n\}$ approximate, as $n\to\infty$, the semigroup
$\{T(t)\}$ in the sense of Definition \ref{9.2}.
\end{theorem}

\begin{proof}
This is exactly the implication $(c)\Rightarrow(a)$ in \cite[Chapter 1,Theorem
7.5]{EK2}.
\end{proof}

The notion of ``extended convergence'' \eqref{tag9.2} turns out to be well
adapted to the application we need. In the context of the paper \cite{BO8},
where we were concerned with the case $\tth=1$, we could manage with a weaker
version of the theorem, based on a stronger assumption: For any $f\in\F$,
$A_n\pi_nf\to Af$. However, in the case of general $\tth>0$ such a weaker
version seems to be insufficient.

\subsection{The Thoma simplex $\Om$ and the embeddings
$\Y_n\hookrightarrow\Om$}\label{9-2} We return to our concrete situation. As
$L_n$ we take the finite-dimensional vector space $\Fun(\Y_n)$ with the
supremum norm, and as $T_n$ we take the Markov transition operator introduced
in Definition \ref{8.1}. Clearly, $T_n$ is a contraction. To define the Banach
space $L$ and the operators $\pi_n:L\to L_n$ we need a preparation.

Let $[0,1]^\infty$ (the infinite-dimensional cube) be the direct product of
countably many copies of the closed unit interval $[0,1]$. We equip
$[0,1]^\infty$ with the product topology; then we get a compact topological
space.

Recall (see Subsection \ref{1-1}) that the Thoma simplex $\Om$ is the subset of
couples $(\al;\be)\in[0,1]^\infty\times[0,1]^\infty$ satisfying the conditions
$$
\al_1\ge\al_2\ge\dots\ge0,\quad
\be_1\ge\be_2\ge\dots\ge0,\quad\sum_i\al_i+\sum_j\be_j\le1.
$$
Clearly, $\Om$ is closed in the topology of $[0,1]^\infty\times[0,1]^\infty$
and hence is a compact topological space.

We set $L=C(\Om)$, the Banach space of real-valued continuous functions on
$\Om$ with the supremum norm.

For any $n$, we define an embedding $\iota_{\tth,n}:\Y_n\to\Om$ in the
following way. Given $\la\in\Y_n$, we divide the shape of $\la$ in the quarter
$(r,s)$ plane into two parts, $\frak A$ and $\frak B$:
$$
\frak A=\{(r,s)\in\shape(\la)\mid s\ge \tth r\}, \quad \frak
B=\{(r,s)\in\shape(\la)\mid s\le \tth r\}.
$$
Let $a_i$ denote the area of the intersection of $\frak A$ with the $i$th row
of boxes, $i=1,2,\dots$. Likewise, let $b_j$ be the area of the intersection of
$\frak B$ with the $j$th column of boxes, $j=1,2,\dots$\,. The sequences
$a_1,a_2,\dots$ and $b_1,b_2,\dots$ are nonincreasing, have finitely many
nonzero terms, and $\sum a_i+\sum b_j=n$. Now, we set
\begin{equation}\label{tag9.3}
\iota_{\tth,n}(\la)=(\al;\be):= \left(\frac{a_1}n,\, \frac{a_2}n,\, \dots;
\frac{b_1}n,\, \frac{b_2}n,\, \dots\right)\in\Om.
\end{equation}

It is easy to check that $\iota_{\tth,n}$ is indeed an embedding. Using it we
define the operator $\pi_n:L\to L_n$, that is, $\pi_n:C(\Om)\to\Fun(\Y_n)$, by
setting
$$
(\pi_n f)(\la)=f(\iota_{\tth,n}(\la)), \qquad f\in C(\Om), \quad \la\in\Y_n\,.
$$
Clearly, $\Vert\pi_n\Vert\le1$. Moreover, the operators $\pi_n$ possess the
following property: For any $f\in L=C(\Om)$,
\begin{equation}\label{tag9.4}
\Vert f\Vert=\lim_{n\to\infty}\Vert\pi_n f\Vert.
\end{equation}
Indeed, this follows from the obvious fact that any open subset in $\Om$ has a
nonempty intersection with $\iota_{\tth,n}(\Y_n)$ for all $n$ large enough.

\subsection{Thoma measures and moment coordinates}\label{9-3}
The content of this subsection is parallel to that of \cite[Subsection
4.4]{BO8} where we considered the particular case $\tth=1$.

Recall that to any point $(\al;\be)\in\Om$ we have assigned a probability
measure $\nu_{\al;\be}$ on the closed interval $[-\tth,1]$, see \eqref{tag1.4}.
The measure $\nu_{\al;\be}$ is called the {\it Thoma measure\/} corresponding
to $(\al;\be)$. Recall also that the moments $q_k=q_k(\al;\be)$ of
$\nu_{\al;\be}$ are given by formula \eqref{tag1.3} and note that the 0th
moment is always equal to 1. As was already said in the Introduction, we call
$q_1, q_2,\dots$ the {\it moment coordinates\/} of the point $(\al;\be)\in\Om$.
Observe that they are continuous functions on $\Om$. Indeed, since $\alpha_i$'s
decrease, the condition $\sum\alpha_i\le 1$ implies $\alpha_i\le i^{-1}$ for
any $i=1,2,\dots$, whence $\alpha_i^{k+1}\le i^{-k-1}$. Similarly,
$\beta_i^{k+1}\le i^{-k-1}$. It follows that the both series in \eqref{tag1.3}
are uniformly convergent on $\Om$, which implies their continuity as functions
on $\Om$. \footnote{This argument substantially relies on the fact that
$k+1\ge2$. Note that the function $(\al;\be)\mapsto \sum\al_i+\sum\be_i$ is not
continuous on $\Om$.}

Let $\frak M([-\tth,1])$ denote the space of probability Borel measures on
$[-\tth,1]$ equipped with the weak topology. Since this topology is determined
by convergence of moments, the assignment $(\al;\be)\mapsto\nu_{\al;\be}$
determines a homeomorphism of the Thoma simplex on a compact subset of $\frak
M([-\tth,1])$.

Note also that the moment coordinates are algebraically independent as
functions on $\Om$. Indeed, this holds even we restrict them on the subset with
all $\be_i$'s equal to 0. It follows that the algebra of polynomials
$\R[q_1,q_2,\dots]$ can be viewed as a subalgebra of the Banach algebra
$C(\Om)$. Since this subalgebra contains 1 and separates points, it is dense in
$C(\Om)$.

Using the correspondence
$$
p^\circ_k \longleftrightarrow q_{k-1}\,, \qquad k=2,3,\dots,
$$
we may identify the algebras $\La^\circ$ and $\R[q_1,q_2,\dots]$, which makes
it possible to realize $\La^\circ$ as a dense subalgebra of $C(\Om)$. In what
follow we will often identify elements $f\in\La^\circ$ and the corresponding
continuous functions $f(\al;\be)$ on $\Om$.

This also enables us to assign to any element $\varphi\in\La$ a continuous
function on $\Om$; according to our convention, this is simply
$\varphi^\circ(\al;\be)$. In equivalent terms, the morphism
$\varphi\mapsto\varphi^\circ(\,\cdot\,)$ is determined by setting
$$
p^\circ_1(\al;\be)\equiv1, \quad p^\circ_k(\al;\be)=\sum_{i=1}^\infty
\al_i^k+(-\tth)^{k-1}\sum_{i=1}^\infty\be_i^k, \qquad k=2,3,\dots,
$$
which is precisely the definition given in \cite{KOO}.

We take $\La^\circ=\R[p^\circ_2,p^\circ_3,\dots]=\R[q_1,q_2,\dots]$ as the
dense subspace $\F\subset C(\Om)$ which has been mentioned in Theorem \ref{1.1}
and in  Theorem \ref{9.3}.

\subsection{Boundary z-measures on $\Om$}\label{9-4}
Fix a couple $(z,z')$ from the principal or complementary series and consider
the system  $\{M^{(n)}_{\tth,z,z'}\}$ of z-measures that we have defined in
Section \ref{5}. Recall that in Subsection \ref{9-2} we have defined the
embeddings $\iota_{\tth,n}:\Y_n\hookrightarrow\Om$ (formula \eqref{tag9.3}).

\begin{theorem}\label{9.4} As $n\to\infty$, the pushforward of the z-measure
$M^{(n)}_{\tth,z,z'}$ under $\iota_{\tth,n}$ weakly converges to a probability
measure $M_{\tth,z,z'}$ on $\Om$.
\end{theorem}

\begin{proof} See \cite[Section 8, proof of Theorem B]{KOO}.
\end{proof}

We call the limit measure $M_{\tth,z,z'}$ on $\Om$ the {\it boundary
z-measure\/}. This agrees with the definitions given in Section \ref{2} after
the statement of Theorem \ref{2.2}. Indeed, as shown in \cite{KOO}, the
topological space $\Om(p^\down)$ corresponding to the system
$p^\down=p^\down_\tth$ can be identified with the Thoma simplex $\Om$ and
$M_{\tth,z,z'}$ is just the measure $M$ that appears in \eqref{tag2.3} when
$\{M^{(n)}\}=M^{(n)}_{\tth,z,z'}$.

Note that the kernel $\mathcal K(\la,\om)$ that appears in formula
\eqref{tag2.3} has the following form:
$$
\mathcal K(\la,\om)=\dim_\tth\la\cdot(\mathcal P^{(\tth)}_\la)^\circ(\al;\be),
\qquad \la\in\Y, \quad \om=(\al;\be)\in\Om,
$$
where $\dim_\tth\la$ is a certain ``$\tth$-version'' of the conventional
dimension function $\dim\la$ (see \cite[\S6]{KOO}) and $\mathcal
P^{(\tth)}_\la\in\La$ is the Jack symmetric function with parameter $\tth$ and
index $\la$.

\subsection{Asymptotics of $\tth$-regular functions}\label{9-5}
Introduce a notation: if $F\in\A_\tth$ is an element of degree $\le m$, that
is, $F\in\A_\tth^{(m)}$ (see Definition \ref{6.2}), then  $[F]_m$ will denote
its highest homogeneous term, which is an element of the quotient space
$\A_\tth^{(m)}/\A_\tth^{(m-1)}$ identified with the $m$th homogeneous component
of the graded algebra $\La$.

\begin{theorem}\label{9.5} Let $F\in\A_\tth$ be an element of degree $\le m$ and
$\varphi=[F]_m\in\La$, as defined above. There exists a constant $C>0$
depending only on $F$, such that for any $n=1,2,\dots$ and any $\la\in\Y_n$ the
following estimate holds
$$
\left|\frac{F(\la)}{n^m}-\varphi^\circ(\iota_{\tth,n}(\la))\right|\le\frac
C{\sqrt n}\,.
$$
\end{theorem}

\begin{proof} This result was proved in \cite[Theorem 8.1]{KOO}, only the
definition of embeddings $\Y_n\hookrightarrow\Om$ employed in \cite{KOO}
slightly differs from that given in Subsection \ref{9-2} above. However, the
difference between the two definitions is unessential. This is seen from the
proof given in \cite{KOO} and the following observation: If $\la\in\Y_n$,
$(\al;\be)=\iota_{\tth,n}(\la)$, and $(\al';\be')$ is the image of $\la$
according to the definition of \cite{KOO}, then
$$
|\al_i-\al'_i|\le \frac1n\,, \quad |\be_i-\be'_i|\le \frac1n\,, \qquad
i=1,2,\dots,
$$
and the number of nonzero coordinates is of order $\sqrt n$. Alternatively, the
reader may simply take the definition of \cite{KOO}. The reason to modify that
definition is purely aesthetic: it is slightly asymmetric with respect to
transposition of rows and columns of a diagram.
\end{proof}

\subsection{The main results}\label{9-6}
Now we are in a position to state and prove the main results of the present
paper. For the reader's convenience let us recall the basic data and
definitions:

$\bullet$ We fix the three basic parameters $\tth,z,z'$, where $\tth>0$ and the
couple $(z,z')$ belongs to the principal or complementary series (Proposition
\ref{5.3} and Definition \ref{5.4}).

$\bullet $ $\Om$ is the Thoma simplex; it is a compact topological space
(Subsection \ref{9-2}).

$\bullet$ $C(\Om)$ is the Banach algebra of continuous real-valued functions on
$\Om$ with supremum norm.

$\bullet$ $\La^\circ$ is the quotient of the algebra $\La$ of symmetric
functions modulo the principal ideal $(p_1-1)\La$ (Definition \ref{8.6});
$\La^\circ$ is embedded into $C(\Om)$ as a dense subalgebra (Subsection
\ref{9-3}).

$\bullet$ $A:\La^\circ\to\La^\circ$ is the operator defined in Corollary
\ref{8.7}, formula \eqref{tag8.7}; we regard $A$ as a densely defined operator
in the Banach space $C(\Om)$ (note that under the identification
$p^\circ_{k+1}=q_k$, \eqref{tag8.7} coincides with \eqref{tag1.5}).

$\bullet$ $T_n: \Fun(\Y_n)\to\Fun(\Y_n)$, $n=1,2,\dots$, are the Markov chain
transition operators introduced in Definition \ref{8.1}.

\begin{theorem}\label{9.6}
{\rm(i)} The operator $A$ is closable and its closure
$\bar A$ generates a strongly continuous contraction semigroup
$\{T(t)\}_{t\ge0}$ in $C(\Om)$.

{\rm(ii)} As $n\to\infty$, the discrete semigroups in the finite-dimensional
spaces\/ $\Fun(\Y_n)$ generated by the contractions $T_n$ approximate the
semigroup $\{T(t)\}_{t\ge0}$ in the sense of Definition \ref{9.2} with the
following choice of the time scaling factors\/{\rm:} $\epsi_n=n^{-2}$.

{\rm(iii)} The semigroup $\{T(t)\}_{t\ge0}$ is a conservative Markov semigroup.
\end{theorem}

Recall that the last property means that each operator $T(t)$ preserves the
constant function $1$ and maps into itself the cone of nonnegative functions in
$C(\Om)$.

\begin{proof} Step 1. As in \eqref{tag9.1}, set $A_n=\epsi^{-1}_n(T_n-1)$. Let us
prove that the operators $A_n$ approximate the operator $A$ in the ``extended''
sense, as explained in Theorem \ref{9.3}.

First of all, it is more convenient to re-define the factors $\epsi_n$
according to \eqref{tag8.1}. Since the new factors are asymptotically
equivalent to $n^{-2}$, this does does not affect the result.

Given $f\in\La^\circ$, fix a natural number $m$ so large that $\deg f\le m$. By
the very definition of $\La^\circ$ and the filtration therein, there exists a
homogeneous element $\varphi\in\La$ of degree $m$ such that $\varphi^\circ=f$.
Next, choose an arbitrary element $F\in\A_\tth^{(m)}$ such that
$[F]_m=\varphi$, see Subsection \ref{9-4} for the notation. Finally, set
$$
f_n=\frac{F_n}{n^m}\,, \qquad n=1,2,\dots\,.
$$
Here, in accordance with the notation of Subsection \ref{7-1}, $F_n$ stands for
the restriction of $F$ to the subset $\Y_n\subset\Y$, so that
$F_n\in\Fun(\Y_n)$.

By virtue of Theorem \ref{9.5}, $\Vert f_n-\pi_n f\Vert\le C/\sqrt n$, so that
$f_n\to f$ in the sense of Definition \ref{9.1}.

Further, set $G=\wt BF$, where the operator $\wt B:\A_\tth\to\A_\tth$ has been
introduced in Theorem \ref{8.2}, and also set
$$
g_n=\frac{G_n}{n^m}\,, \qquad n=1,2,\dots\,.
$$
By virtue of Theorem \ref{8.2}, $G\in\A_\tth^{(m)}$ and
$[G]_m=B[F]_m=B\varphi$, where the operator $B:\La\to\La$ has been introduced
in Definition \ref{8.3}. Applying again Theorem \ref{9.5} we get $g_n\to g$,
where $g:=(B\varphi)^\circ$.

On the other hand, Corollary \ref{8.7} says that
$(B\varphi)^\circ=A\varphi^\circ=Af$. Therefore, $A_nf\to Af$. Thus, we have
proved the required ``extended'' convergence $A_n\to A$.

\medskip

Step 2. Let us prove that $A$ is a dissipative operator, that is, for any $s>0$
and any $f\in\La^\circ\subset C(\Om)$ we have the inequality
$\Vert(s-A)f\Vert\ge s\Vert f\Vert$.

Indeed, according to the result of step 1, there exist $f_n\in\Fun(\Y_n)$ such
that $f_n\to f$ and $A_nf_n\to Af$. Since the operators $T_n$ are contractions,
the operators $A_n$ are dissipative. Consequently, $\Vert(s-A_n)f_n\Vert\ge
s\Vert f_n\Vert$.

On the other hand, because of \eqref{tag9.4}, in our situation, convergence of
vectors (in the sense of Definition \ref{9.1}) implies convergence of their
norms. Therefore, we may pass to the limit in the above inequality for $f_n$
and get the desired inequality for $f$.

\medskip

Step 3. As is seen from \eqref{tag8.7}, the operator $A$ does not raise degree
in the sense of the canonical filtration of the algebra $\La^\circ$ (this is
also obvious because $B$ preserves the graduation in $\La$). Let
$\La^{\circ(m)}\subset\La^\circ$ stand for the subspace of elements of degree
$\le m$. Each such subspace has finite dimension and is invariant under $A$.
Because $A$ is dissipative (step 2),  the operator $s-A$ maps $\La^{\circ(m)}$
onto itself for any $s>0$ and any $m$. Since the subspaces $\La^{\circ(m)}$
form an ascending chain and their union is the whole space $\La^\circ$, we
conclude that $s-A$ maps $\La^\circ$ onto itself. The combination of this
property and the dissipativity property entails that $A$ is closable and its
closure $\bar A$ serves as the generator of a strongly continuous contraction
semigroup $\{T(t)\}_{t\ge0}$ in $C(\Om)$: this is a version of the
Hille--Yosida theorem, see, e.~g., Theorem 2.12 in \cite{EK2}. Thus, we have
checked claim {\rm(i)} of the theorem.

\medskip

Step 4. Now claim {\rm(ii)} immediately follows from Theorem \ref{9.3}. Indeed,
we have just established the existence of the semigroup $\{T(t)\}$, and the
validity of the hypothesis of Theorem \ref{9.3} has been verified on step 1.

\medskip

Step 5. Let us check claim {\rm(iii)}. Since the constant term of the
differential operator $A$ vanishes, we have $A1=0$, which implies $T(t)1=1$ for
all $t\ge0$. Therefore, the semigroup is conservative.

Let us to prove that if $f\in C(\Om)$ is nonnegative then so is $T(t)f$.

Obviously, $\pi_nf$ is a nonnegative function for any $n$. Since $T_n$ is the
transition operator of a Markov chain, $T_n^m\pi_nf$ is a nonnegative function
on $\Y_n$ for any natural number $m$. In particular, $T^{[t\epsi^{-1}_n]}\pi_n
f\ge0$.

On the other hand, because claim {\rm(ii)} has already been established, we
know that $\de_n:=\Vert T_n^{[t\epsi^{-1}_n]}\pi_nf-\pi_n T(t)f\Vert\to 0$ as
$n\to 0$ (see Definition \ref{9.2}). Therefore, $\pi_n T(t)f\ge-\de_n$ on
$\Y_n$. In other words, $T(t)f\ge-\de_n$ on the subset
$\iota_{\tth,n}(\Y_n)\subset\Om$. As pointed out in the very end of Subsection
\ref{9-2}, this finite subset becomes more and more dense in $\Om$ as
$n\to\infty$. Since $\de_n\to0$ and the function $T(t)f$ is continuous, it is
nonnegative on the whole $\Om$.

This concludes the proof.
\end{proof}

By a well-known general result (see \cite[Chapter 4, Theorem 2.7]{EK2}), the
Markov semigroup $\{T(t)\}$ constructed in Theorem \ref{9.6} gives rise to a
strong Markov process in $\Om$ with c\`adl\`ag sample paths. Let us denote this
process by $\ome_{\tth,z,z'}(t)$. Actually, due to the knowledge of the
explicit form of the pre-generator $A$ (formula \eqref{tag8.7}) one can get a
stronger result:

\begin{theorem}\label{9.7} The Markov process $\ome_{\tth,z,z'}(t)$ has continuous
sample paths.
\end{theorem}

\begin{proof} The argument is exactly the same as in the case $\tth=1$, see
\cite[Theorem 8.1]{BO8}. One shows that any smooth cylinder function in the
moment coordinates $q_1=p^\circ_2,q_2=p^\circ_3,\dots$ enters the domain of the
generator $\bar A$ (\cite[Corollary 7.4]{BO8}). Using this, one can verify the
Dynkin--Kinney condition.
\end{proof}

The next results are related to the boundary z-measure $M_{\tth,z,z'}$ defined
in Subsection \ref{9-4}. Below the angular brackets
$\langle\,\cdot\,,\,\cdot\,\rangle$ denote the pairing between functions and
measures. Consider the inner product in $C(\Om)$ determined by
\begin{equation}\label{tag9.5}
(f,g)=\langle fg, M_{\tth,z,z'}\rangle.
\end{equation}

\begin{lemma}\label{9.8}
{\rm(i)} If $f\in C(\Om)$, $f_n\in\Fun(\Y_n)$, and $f_n\to f$ in the sense of
Definition \ref{9.1}, then
$$
\langle f_n\,, M^{(n)}_{\tth,z,z'}\rangle \to \langle f\,,
M_{\tth,z,z'}\rangle.
$$

{\rm(ii)} If $f_n\to f$ and $g_n\to g$, where $f,g\in C(\Om)$ and
$f_n,g_n\in\Fun(\Y_n)$, then $f_ng_n\to fg$.
\end{lemma}

\begin{proof} {\rm(i)} Set $\wt M^{(n)}_{\tth,z,z'}
=\iota_{\tth,n}(M^{(n)}_{\tth,z,z'})$; this is a probability measure on $\Om$.
By Theorem \ref{9.4}, the measures $\wt M^{(n)}_{\tth,z,z'}$ weakly converge to
the measure $M_{\tth,z,z'}$ as $n\to 0$. Therefore,
$$
\langle f\,, \wt M^{(n)}_{\tth,z,z'}\rangle \to \langle f\,,
M_{\tth,z,z'}\rangle.
$$

On the other hand,
$$
\langle f\,, \wt M^{(n)}_{\tth,z,z'}\rangle=\langle \pi_n f\,,
M^{(n)}_{\tth,z,z'}\rangle.
$$
Further, the assumption $f_n\to f$ just means $\Vert f_n-\pi_n f\Vert\to 0$, so
that
$$
\langle f_n\,, M^{(n)}_{\tth,z,z'}\rangle\,-\,\langle \pi_n f\,,
M^{(n)}_{\tth,z,z'}\rangle \to0.
$$
This proves the claim.

{\rm(ii)} We know that $\Vert f_n-\pi_n f\Vert\to0$ and $\Vert g_n-\pi_n
g\Vert\to0$, and we  have to check  that
$$
\Vert f_ng_n-\pi_n (fg)\Vert \to0.
$$
Since $\pi_n(fg)=(\pi_n f)(\pi_n g)$ and all the functions under consideration
are uniformly bounded, this is obvious.
\end{proof}

\begin{theorem}\label{9.9} {\rm(i)} The pre-generator $A:\La^\circ\to\La^\circ$ is
symmetric with respect to the inner product\/ \eqref{tag9.5} in the space
$\La^\circ$ and can be diagonalized in an appropriate orthogonal basis.

{\rm(ii)} The spectrum of $A$ is $\{0\}\cup\{-\si_m: m=2,3,\dots\}$ where
$$
\si_m=m(m-1+\tth^{-1}zz'), \qquad m=2,3,\dots,
$$
the eigenvalue\/ $0$ is simple, and the multiplicity of $-\si_m$
equals the number of partitions of $m$ without parts equal to $1$.
\footnote{Denoting by $\frak p(m)$ the number of {\it all\/}
partitions of $m$, the multiplicity in question can be written as
$\frak p(m)-\frak p(m-1)$.}
\end{theorem}

Note that this result also describes the spectrum of the closure of $A$ in the
Hilbert space $L^2(\Om, M_{\tth,z,z'})$.

\begin{proof}
{\rm(i)} We have to prove that for any $f,g\in\La^\circ$
\begin{equation}\label{tag9.6}
\langle (Af)g-f(Ag), \,M_{\tth,z,z'}\rangle=0.
\end{equation}
As shown on step 1 of the proof of Theorem \ref{9.6}, there exist sequences
$\{f_n\in\Fun(\Y_n)\}$ and $\{g_n\in\Fun(\Y_n)\}$ such that
\begin{equation}\label{tag9.7}
f_n\to f, \quad A_nf_n\to Af, \quad g_n\to g, \quad A_ng_n\to Ag
\end{equation}
in the sense of Definition \ref{9.1}.

On the other hand, observe that the transition operator
$T_n:\Fun(\Y_n)\to\Fun(\Y_n)$ is symmetric with respect to the inner product in
$\Fun(\Y_n)$ given by the formula similar to \eqref{tag9.5} but with
$M^{(n)}_{\tth,z,z'}$ instead of $M_{\tth,z,z'}$. Indeed, this follows from the
fact that $M^{(n)}_{\tth,z,z'}$ is the symmetrizing measure (Proposition
\ref{2.5}). Therefore, $A_n$ is also symmetric and hence
\begin{equation}\label{tag9.8}
\langle (A_nf_n)g_n-f_n(A_ng_n), \,M^{(n)}_{\tth,z,z'}\rangle=0.
\end{equation}

Now we apply Lemma \ref{9.8}. By virtue of its claim {\rm(ii)}, \eqref{tag9.7}
implies
$$
(A_nf_n)g_n-f_n(A_ng_n) \,\to\,(Af)g-f(Ag),
$$
and then claim {\rm(i)} makes it possible to pass to the limit in
\eqref{tag9.8}, which gives \eqref{tag9.6}.

The existence of an orthogonal eigenbasis for $A$ follows from the
fact that $A$ preserves each of the finite-dimensional subspaces
$\La^{\circ(m)}$ forming the filtration of $\La^\circ$.

{\rm(ii)} Set
$$
E=\sum_{k\ge2}kp^\circ_k\frac{\pd}{\pd p^\circ_k}\,.
$$
As seen from \eqref{tag8.7}, $A$ can be represented as the sum of the operator
\begin{equation}\label{tag9.9}
-E(E-1+\tth^{-1}zz')
\end{equation}
and a rest term which has degree $\le-1$. The operator \eqref{tag9.9} is
diagonalized in the (non-orthogonal) basis formed by $1$ and the monomials in
the generators $p^\circ_2,p^\circ_3,\dots$, and has precisely the spectrum
indicated in the statement of the proposition. The rest term, obviously, does
not affect the spectrum.
\end{proof}

\begin{theorem}\label{9.10}
{\rm(i)} The Markov process $\ome_{\tth, z,z'}(t)$  has the boundary measure
$M_{\tth, z,z'}$ as a unique stationary distribution.

{\rm(ii)} It is also a symmetrizing measure.

{\rm(iii)} The process is ergodic in the sense that for any $f\in C(\Om)$,
$$
\lim_{t\to+\infty}\Vert T(t)f-\langle f, M_{\tth,z,z'}\rangle1\Vert=0,
$$
where  $1$ is the constant function equal to one.
\end{theorem}

\begin{proof} This result relies on Theorem \ref{9.9}: the argument is exactly the
same as in the proof of Theorem \ref{8.3} from \cite{BO8}.
\end{proof}

\begin{remark}\label{9.11} Recall the notation $\F=\La^\circ$ for the domain of the
pre-generator $A$ of the Markov process $\ome_{\tth, z,z'}(t)$. The
pre-Dirichlet form on $\F\times\F$ corresponding to the pre-generator can be
written in the following way
$$
-\int_\Om (AF)(\al;\be)\,G(\al;\be)M_{\tth,z,z'}(d\al
d\be)=\int_{\Om}\Ga(F,G)(\al;\be)M_{\tth,z,z'}(d\al d\be),
$$
where the $F,G\in\F$ and $\Ga(\,\cdot\,,\,\cdot\,)$ (the ``square field
operator'') is a symmetric bilinear map $\F\times\F\to\F$ which {\it does not
depend\/} on the parameters $\tth,z,z'$:
$$
\Ga(F,G)=\sum_{i,j=1}^\infty (i+1)(j+1)(q_{i+j}-q_iq_j)\frac{\pd F}{\pd
q_i}\frac{\pd G}{\pd q_j} =\sum_{k,l=2}^\infty kl(p^\circ_{k+l-1}-p^\circ_k
p^\circ_l)\frac{\pd F}{\pd p^\circ_k}\frac{\pd G}{\pd p^\circ_l}\,.
$$
The proof is exactly the same as in \cite[Theorem 8.4]{BO8}; it relies on the
fact that the coefficients of the second derivatives in \eqref{tag1.5} or
\eqref{tag8.7} do not depend on the parameters.
\end{remark}

\begin{remark}\label{9.12} Here is a complement to the remark made in Subsection
\ref{1-2} about the possibility to degenerate the pre-generator $A$ given by
formula \eqref{tag1.5} (or, equivalently, by \eqref{tag8.7}) to the
Ethier--Kurtz--Schmuland operator \eqref{tag1.7} in the limit regime
\eqref{tag1.6}. Recently Petrov \cite{Pe1} found a two-parameter generalization
of the Ethier--Kurtz diffusion associated to Pitman's two-parameter
generalization $P(\al,\tau)$ of the Poisson--Dirichlet distribution (here
$\al\in[0,1)$ is the additional parameter and $\tau$ should be strictly greater
than $-\al$). \footnote{I recall that I am writing $PD(\al,\tau)$ instead of
the conventional $PD(\al,\tth)$ to avoid a conflict of notation.} The
corresponding two-parameter extension of the operator \eqref{tag1.7} has the
form
\begin{equation}\label{tag9.10}
\sum_{i,j\ge1}(i+1)(j+1) (q_{i+j}-q_i q_j)\frac{\pd^2}{\pd q_i \pd q_j}\;+\;
\sum_{i\ge1}(i+1)\big[(i-\al) q_{i-1}-(i+\tau) q_i\big]\frac{\pd}{\pd q_i}\,,
\end{equation}
see \cite[formula (16)]{Pe1}.  Now, observe that this more general operator can
also be obtained by degeneration from \eqref{tag1.5}: to achieve this we have
to impose the following conditions on the asymptotics of our triple
$(\tth,z,z')$:
\begin{equation}\label{tag9.11}
\tth\to0, \quad zz'\to 0, \quad \tth^{-1}zz'\to\tau, \quad z+z'\to-\al.
\end{equation}
Moreover, one can show that in this limit regime, the down transition
probabilities $p^\down_\tth(\la,\mu)$, the up transition probabilities
$p^\up_{\tth,z,z'}(\la,\nu)$, and the weights $M^{(n)}_{\tth,z,z'}(\la)$
converge to the respective quantities considered in \cite{Pe1}. However, for
$\al\ne0$, the limit regime \eqref{tag9.11} is {\it incompatible\/} with the
restrictions on $(z,z')$ that ensure positivity of the up transition
probabilities and the z-measures (see \cite[Proposition 2.3]{BO5}). That is, if
we wish to perform the limit \eqref{tag9.11} with a nonzero $\al$, then we
inevitably have to admit those $(z,z')$'s for which the pre-limit quantities
$p^\up_{\tth,z,z'}(\la,\nu)$ and $M^{(n)}_{\tth,z,z'}(\la)$ can take negative
or even complex values, which makes the limit transition purely formal.
\end{remark}


\begin{thebibliography}{XXXX}

\bibitem[B1]{B1}  A. Borodin,  Multiplicative central measures on the Schur
graph. J. Math. Sci. (New York)  {\bf96}  (1999), no.~5,  3472--3477.


\bibitem[B2]{B2}  A. Borodin  Harmonic analysis on the infinite symmetric
group and the Whittaker kernel.  St.~Petersburg Math. J.  {\bf12}  (2001),
no.~5,  733--759.


\bibitem[BO1]{BO1}  A. Borodin and G. Olshanski,  Point processes and the
infinite symmetric group.  Math. Research Lett. {\bf 5}  (1998),  799--816;
arXiv: math.RT/9810015.


\bibitem[BO2]{BO2}  A. Borodin and G. Olshanski,  Distributions on
partitions, point processes and the hypergeometric kernel.  Comm. Math. Phys.
{\bf 211}  (2000),  335--358;  arXiv: math.RT/9904010.


\bibitem[BO3]{BO3}  A. Borodin and G. Olshanski.  Harmonic functions on
multiplicative graphs and interpolation polynomials.  Electronic J. Comb. {\bf
7} (2000),  paper \#R28; arXiv: math/9912124.


\bibitem[BO4]{BO4}  A. Borodin and G. Olshanski,  Z--Measures on partitions,
Robinson--Schensted--Knuth correspondence, and $\beta=2$ random matrix
ensembles.  In: Random matrix models and their applications (P.~M.~Bleher and
A.~R.~Its, eds). Mathematical Sciences Research Institute Publications {\bf
40}, Cambridge Univ. Press,  2001,  pp. 71--94; arXiv: math/9905189.


\bibitem[BO5]{BO5}  A. Borodin and G. Olshanski,  Z-measures on partitions
and their scaling limits.  European J. Comb.  {\bf26}  (2005),  no. 6,
795--834.


\bibitem[BO6]{BO6}  A. Borodin and G. Olshanski,  Markov processes on
partitions.  Prob. Theory and Related Fields  {\bf 135}  (2006),  no.~1,
84--152; arXiv: math-ph/0409075.


\bibitem[BO7]{BO7}  A. Borodin and G. Olshanski,  Stochastic dynamics related
to Plancherel measure on partitions.   In: Representation Theory, Dynamical
Systems, and Asymptotic Combinatorics (V.~Kaimanovich and A.~Lodkin, eds).
Amer. Math. Soc. Translations, Series 2: Advances in the Mathematical Sciences,
{\bf217}, 2006, pp. 9--21.



\bibitem[BO8]{BO8}  A. Borodin and G. Olshanski,  Infinite-dimensional
diffusions as limits of random walks on partitions.  Probab. Theory Rel. Fields
{\bf144} (2009), no.~1,  281--318; arXiv:0706.1034.


\bibitem[Dy]{Dy}  F. J. Dyson,  A Brownian-motion model for the eigenvalues
of a random matrix.  J. Math. Phys. {\bf 3}  (1962),  1191--1198.


\bibitem[EK1]{EK1}  S. N. Ethier and T. G. Kurtz,  The
infinitely-many-neutral-alleles diffusion model.  Adv. Appl. Prob.  {\bf13}
(1981),  429--452.

\bibitem[EK2]{EK2}  S. N. Ethier and T. G. Kurtz,   Markov processes ---
Characterization and convergence.   Wiley--Interscience.   New York,  1986.

\bibitem[FW]{FW}  P. J. Forrester and S. O. Warnaar,  The importance of the
Selberg integral.  Bull. Amer. Math. Soc.  {\bf 45}  (2008), no.~4,  489--534.


\bibitem[Fu1]{Fu1}  J. Fulman,  Stein's method and Plancherel measure of the
symmetric group.  Trans. Amer. Math. Soc. {\bf 357}  (2005),  555--570.


\bibitem[Fu2]{Fu2}  J. Fulman,  Stein's method and random character ratios.
Trans. Amer. Math. Soc.  {\bf 360} (2008),  no.~7,  3687--3730.


\bibitem[Fu3]{Fu3}  J. Fulman,  Commutation relations and Markov
chains.  Probab. Theory Rel. Fields {\bf144} (2009), 99--136; arXiv:0712.1375.


\bibitem[HS]{HS}  G. Heckman and H. Schlichtkrull,   Harmonic analysis and
special functions on symmetric spaces.   Perspectives in Mathematics {\bf16}.
Academic Press, Inc., San Diego, CA, 1994.


\bibitem[IO]{IO}  V. Ivanov and G. Olshanski,  Kerov's central limit theorem
for the Plancherel measure on Young diagrams.   In:  Symmetric functions 2001.
Surveys of developments and perspectives. Proc. NATO Advanced Study Institute
(S.~Fomin, editor), Kluwer, 2002, pp. 93--151.



\bibitem[Ka]{Ka}  K. W. J. Kadell,   The Selberg-Jack symmetric functions.
 Adv. Math.  {\bf 130}  (1997), no.~1,  33--102.


\bibitem[Ke1]{Ke1}  S. V. Kerov.  Generalized Hall--Littlewood symmetric
functions and orthogonal polynomials.   In: Representation theory and dynamical
systems   Advances in Soviet Math. {\bf9}.   Amer. Math. Soc.,  1992, pp.
67--94.


\bibitem[Ke2]{Ke2}  S. V. Kerov,  Transition probabilities of continual Young
diagrams and the Markov moment problem.  Funct. Anal. Appl. {\bf 27}  (1993)
no.~2,  104--117.


\bibitem[Ke3]{Ke3}  S. V. Kerov,  The boundary of Young lattice and random
Young tableaux.   In: Formal power series and algebraic combinatorics (New
Brunswick, NJ, 1994)   DIMACS Ser. Discrete Math. Theoret. Comput. Sci.
{\bf24}. Amer. Math. Soc.   Providence, RI, 1996, pp.  133--158.


\bibitem[Ke4]{Ke4}  S. V. Kerov,  Anisotropic Young diagrams and Jack
symmetric functions.  Funct. Anal. Appl. {\bf 34} (2000) no.~1,  45--51;
arXiv:math/9712267.


\bibitem[Ke5]{Ke5}  S. V. Kerov,  Interlacing measures.  In:
Kirillov's seminar on representation theory (G.~Olshanski, ed.), Amer. Math.
Soc. Transl. Ser. 2, {\bf181}, Amer. Math. Soc., Providence, RI, 1998, pp.
35--83.


\bibitem[Ke6]{Ke6}  S. V. Kerov,   Asymptotic representation theory of the
symmetric group and its applications in analysis.   Amer. Math. Soc.,
Providence, RI, 2003


\bibitem[KeO]{KeO}  S. Kerov and G. Olshanski,  Polynomial
functions on the set of Young diagrams.  Comptes Rendus Acad. Sci. Paris,
S\'er. I, {\bf 319}  (1994),  121--126.


\bibitem[KOO]{KOO}  S. Kerov, A. Okounkov, G. Olshanski,  The boundary of
Young graph with Jack edge multiplicities.  Intern. Math. Res. Notices  (1998),
no.~4, 173--199; arXiv: q-alg/9703037.


\bibitem[KOV1]{KOV1}  S. Kerov, G. Olshanski, and A. Vershik,  Harmonic
analysis on the infinite symmetric group. A deformation of the regular
representation. Comptes Rendus Acad. Sci. Paris, S\'er. I,  {\bf316}  (1993),
773--778.


\bibitem[KOV2]{KOV2}  S. Kerov, G. Olshanski, and A. Vershik,  Harmonic
analysis on the infinite symmetric group.  Invent. Math. {\bf 158}  (2004),
551--642; arXiv: math.RT/0312270.


\bibitem[Ki1]{Ki1}  J. F. C. Kingman,  The representation of partition
structures.  J. London Math. Soc  {\bf 18}  (1978),  no.~2,  374--380.


\bibitem[Ki2]{Ki2}  J. F. C. Kingman,  Poisson processes.   Clarendon
Press, Oxford, 1993.


\bibitem[LT]{LT}  A. Lascoux and J.-Y. Thibon,  Vertex operators
and the class algebras of symmetric groups.   J. Math. Sci. (N. Y.)  {\bf 121}
(2004), no.~3,  2380--2392; arXiv:math/0102041.



\bibitem[Ma]{Ma}  I. G. Macdonald,  Symmetric functions and Hall polynomials,
 2nd edition, Oxford University Press,  1995.


\bibitem[Ok]{Ok}  A. Okounkov,  The uses of random partitions. In:
XIVth International Congress on Mathematical Physics,   World Sci. Publ.,
Hackensack, NJ, 2005, pp. 379--403


\bibitem[OO]{OO}  A. Okounkov and G. Olshanski,  Shifted Jack
polynomials, binomial formula, and applications.  Math. Research Lett. {\bf 4}
(1997),  69--78.



\bibitem[Ol]{Ol}  G. Olshanski,  Point processes related to the infinite
symmetric group. In: The orbit method in geometry and physics: in honor of
A.~A.~Kirillov (Ch.~Duval, L.~Guieu, V.~Ovsienko, eds.), Progress in
Mathematics {\bf 213}, Birkh\"auser, 2003, pp. 349--393; arXiv:math/9804086.



\bibitem[Pe1]{Pe1}  L. Petrov,  Two--parameter family of diffusion processes
in the Kingman simplex.  Funct. Anal. Appl. {\bf43} (2009), no. 4;
arXiv:0708.1930.


\bibitem[Pe2]{Pe2}  L. Petrov,  Random walks on strict partitions, J. Math.
Sciences (New York), to appear; arXiv:0904.1823.


\bibitem[Sch]{Sch}  B. Schmuland,  A result on the infinitely many neutral
alleles diffusion model. J. Appl. Prob.  {\bf28} (1991),  253--267.


\bibitem[SV]{SV}  A. N. Sergeev and A. P. Veselov,  Generalised
discriminants, deformed Calogero--Moser--Sutherland operators and super-Jack
polynomials.  Adv. Math.  {\bf192}  (2005), no.~2,  341--375.


\bibitem[Sp]{Sp}  H. Spohn,  Interacting Brownian particles: a
study of Dyson's model. In: Hydrodynamic Behavior and Interacting Particle
Systems, G. Papanicolaou (ed), IMA Volumes in Mathematics and its Applications,
9, Berlin, Springer-Verlag, 1987, pp. 151--179.


\bibitem[Sta]{Sta}  R. P. Stanley,  Irreducible symmetric group characters of
rectangular shape.  S\'emin. Lothar. Combin.  {\bf50} (2003),  Article B50d, 11
pp.


\bibitem[Str]{Str} E. Strahov, Z-measures on partitions related to the infinite
Gelfand pair $(S(2\infty), H(\infty))$, arXiv:0904.1719.


\bibitem[T]{T}  E. Thoma,  Die unzerlegbaren, positive-definiten
Klassenfunktionen der abz\"ahlbar unendlichen, symmetrischen Gruppe.
Math.~Zeitschr. {\bf 85}  (1964),  40--61.


\bibitem[VK]{VK}  A. M. Vershik and S. V. Kerov,  Asymptotic theory of
characters of the symmetric group.  Funct. Anal. Appl. {\bf 15}  (1981),
246--255.

\end{thebibliography}
\end{document}